\documentclass[final,reqno,twoside,a4paper,11pt]{amsart}

\usepackage{mltools}
\usepackage{verbatim}
\usepackage{mlmath62,mlthm10-REAR,mllocal}
\usepackage{upref,amsmath,amssymb,amsthm,mathrsfs}
\usepackage{tikz,graphics,latexsym}
\usepackage{times}
\usepackage{amsmath}
\usepackage{amssymb}
\usepackage{amsthm}
\usepackage{amsfonts}
\usepackage[all]{xy}
\usepackage{xcolor}
\usepackage[utf8]{inputenc}

\usepackage[scaled]{helvet} % ss
\usepackage{courier} % tt
\usepackage[mathbf]{euler}
\usepackage{mllocal}
\usepackage{pgfplots}
\numberwithin{equation}{section}

\begin{document}
\begin{comment}
\documentclass[12pt]{article}
\setlength{\oddsidemargin}{.35cm}
\setlength{\evensidemargin}{.35cm} \setlength{\marginparsep}{1mm}
\setlength{\marginparwidth}{.8cm} \setlength{\textwidth}{15.5cm}
\setlength{\topmargin}{-1.3cm}
\setlength{\textheight}{24cm}
\setlength{\headheight}{.1in}

%  Theorems and Corollary in the Introduction
\newtheorem{Thm}{Theorem}
\newtheorem{Cor}[Thm]{Corollary}
%  Theorems, etc.
\newtheorem{theorem}{Theorem}[section]
\newtheorem{lemma}[theorem]{Lemma}
\newtheorem{corollary}[theorem]{Corollary}
\newtheorem{proposition}[theorem]{Proposition}
\newtheorem{examples}[theorem]{Examples}
%  Definitions, Remarks.
\theoremstyle{definition}
\newtheorem{observation}[theorem]{Remark}
\newtheorem{definition}[theorem]{Definition}
\def \fim {{\hfill $\blacksquare$}}

\newenvironment{prf}
{\begin{trivlist} \item[\hskip \labelsep {\bf Proof}\hspace*{3
mm}]} {\hfill\rule{2.5mm}{2.5mm} \end{trivlist}}

\newenvironment{equationth}{\stepcounter{theorem}\begin{equation}}{\end{equation}}

\def\Z{\mathbb Z}
\def\C{\mathbb C}
\def\R{\mathbb R}
\def\sph{\mathbb S}
\def\B{\mathbb B}
\def\0{\underline 0}
\def\M{\mathcal M}
\def\e{\varepsilon}
\end{comment}

\title {\bf Relative Bruce-Roberts number and Chern obstruction}

\vspace{1cm}
\author{Bárbara K. Lima Pereira, Maria Aparecida Soares Ruas, Hellen Santana}

\maketitle

\begin{abstract}

Let $(X,0)$ be the germ of an equidimensional analytic set in $(\mathbb C^n,0)$ and $f=(f_1,f_2)$ a map-germ into the plane
defined on $X.$
In this work, we investigate topological invariants associated to the pair $(f,X),$  among them, the Euler obstruction of $f,$
$Eu_{f,X}(0),$ and under convenient assumptions, the Chern number of families of differential forms associated to $f.$ The topological information provided
by these invariants is useful, although difficult to calculate. The aim of the paper is to introduce the Bruce-Roberts and the relative Bruce-Roberts
numbers as useful algebraic tools to capture the topological information giving by the Euler obstruction and the Chern numbers. Closed formulas are given when $X,\, X\cap f_2^{-1}(0),\,  X\cap f_2^{-1}(0)\cap f_1^{-1}(0)$ are ICIS. In the last section, for a 2-dimensional ICIS
$(X,0) \subset (\mathbb C^n,0),$  we apply our results to give an alternative description for the number of cusps $c(f|_X)$ of an stabilization of an $\mathcal A$-finite map-germ $f=(f_1, f_2): (X,0) \to (\mathbb C^2,0).$ A formula for $c(f|_X)$ was first given in [21].

\end{abstract}

\section*{Introduction}

Let $\mathcal O_n$ be the local ring of  analytic function germs at the origin.
 The Milnor number of a function germ  $f: (\mathbb C^n,0) \to (\mathbb C,0),$  denoted $\mu(f),$  is equal to the dimension as
$\mathbb C$ - vector space of the quotient $\mathcal O_n/Jf.$

In 1968, Milnor proved that, if $f \in \mathcal O_n$ has an isolated singularity, the Milnor fiber $f^{-1}(\delta) \cap B_{\epsilon},$ where $\delta$ is a regular value of
$f$, has the homotopy type of a wedge of $\mu(f)$ spheres of dimension $n-1.$ It is also known that $\mu(f)$ is  the number of Morse points of a Morsification of
$f$ in a neighborhood of the origin.

Now, suppose $(X,0)$ is the germ of an equidimensional analytic variety in $(\mathbb C^n,0)$ and $f$ an analytic function germ defined on $X.$
For functions with (stratified) isolated singularity on $X,$ there are in the literature different approaches to a generalization of the Milnor number. We refer in this work  to the {\it Euler obstruction} of $f$ at the origin, denoted $Eu_{f,X}(0),$ which was defined by Brasselet, Massey, Parameswaran and Seade in \cite{BMPS} (see also \cite{STV}). This invariant preserves some properties of the Milnor number, for instance, $Eu_{f,X}=(-1)^d n_{reg},$ where $n_{reg}$ is the number of Morse points in $X_{reg}$ in a stratified Morsification of $f,$ where $X_{reg}$ is the regular part of $X$.

Other numbers that generalize the Milnor number were defined by Bruce and Roberts in \cite{bruce1988critical} and are given by $\mu_{BR}(f,X)={\dim}_{\mathbb C}\mathcal O_n/df(\Theta_X)$ and $\mu_{BR}^{-}(f,X)={\dim}_{\mathbb C}\mathcal O_n/(df(\Theta_X)+I_X),$ where $df(\Theta_X)$ is the ideal generated by  the derivatives of $f$ along logarithmic vector fields of an analytic variety $X$ and $I_X$ is the ideal of analytical function germs which vanish on $(X,0)$. These invariants have been called Bruce-Roberts number and relative Bruce-Roberts number, respectively, and in the particular case when $(X,0)$ is an Isolated Complete Intersection Singularity (ICIS), there are formulas to compute them.
%\begin{align*}
%\mu_{BR}(f,X)=\mu(f)+\mu(X\cap f^{-1}(0),0)+\mu(X,0)-\tau(X,0)+&\dim_{\C}\frac{\mathcal{O}_{n}}{Jf+I_{X}}+\\
%&\dim_{\C}Tor_{1}\left(\frac{\mathcal{O}_{n}}{I_{X}},\frac{\mathcal{O}_{n}}{I_{X}}\right)
%\end{align*}
%and the particular case for an Isolated Hypersurface Singularity (IHS) the previous equality reduces to
%$$
%\mu_{BR}(f,X)=\mu(f)+\mu(X\cap f^{-1}(0),0)+\mu(X,0)-\tau(X,0).
%$$

We indicate the following %}and functions $f$ with isolated singularity on $X.$  In
recent papers for these formulas, \cite{juanjobrunatomazella2}, when $(X,0)$ is a weighted homogeneous hypersurface with isolated singularity, \cite{lima2021relative} and \cite{kourliouros2021milnor} for any Isolated Hypersurface Singularity (IHS), and \cite{lima2022} for any ICIS (see also \cite{bivia2022bruce, bivia2020mixed}). The geometric interpretation of the Bruce-Roberts number is also given in terms of the number of critical points of a Morsification of $f,$ however the precise description is related to properties of the logarithmic characteristic variety of $(X,0)$. This variety, denoted by $LC(X)$, is a subvariety of the cotangent bundle of $(\mathbb{C}^n,0)$, and it is Cohen-Macaulay at $(0,df(0))$, if and only if the Bruce-Roberts number of the function germ $f$ with respect to $(X,0)$ is equal to the number of critical points of a Morsification $F$ of $f$, counted with multiplicities, see \cite[corollary 5.8]{bruce1988critical} for details.
However, when the variety $(X,0)$ has codimension higher than one we know that $LC(X)$ is not Cohen-Macaulay \cite[Proposition 5.10]{bruce1988critical}. Then one can consider the relative logarithmic characteristic variety, $LC(X)^{-}$, which is Cohen-Macaulay at $(0,df(0))$, if and only if, the relative Bruce-Roberts number of $f$ with respect to $(X,0)$ counts the number of critical points of a Morsification $F$ of $f$ in $(X,0)$, counted with multiplicities, see \cite[Proposition 5.11]{bruce1988critical}.  In the particular case when $(X,0)$ is an ICIS we know that $LC(X)^{-}$ is Cohen-Macaulay,  \cite[Theorem 3.1]{lima2022}.
%It denotes the closure of the subset of $LC(X)$ obtained by deleting the component coming from the logarithmic stratum $X_{0}=\C^{n}\setminus X$. % The relative Bruce-Roberts number is given by $\mu_{BR}^{-}(f,X)=\dim_{\C}\frac{\mathcal{O}_{n}}{df(\Theta_{X})+I_{X}}$.
 %Cohen-Macaulay at $(0,df(0))$, if and only if, the relative Bruce-Roberts number of $f$ with respect to $(X,0)$ also counts the number of critical points of a Morsification $F$ of $f$ in $(X,0)$, counted with multiplicities, see\cite[Proposition 5.11]{bruce1988critical}.  In the particular case when $(X,0)$ is an ICIS we know that $LC(X)^{-}$ is Cohen-Macaulay,  \cite[Theorem 3.1]{lima2022}.
%(RELATIVO?)}

%Agora precisa falar que sob certas condições, o número de Bruce-Roberts relativo é dado pela formula da soma, com multiplicidade,
%dos número de pontos criticos......
%was defined by Bruce and Roberts in \cite{}.  Bruce-Roberts number of $f,$  $\mu_{BR}(X,f).$

\vspace{0.2cm}

The goal in this paper is to present a generalization of these results  for map-germs defined on singular varieties, with focus in the case $f: (X,0) \to  (\mathbb C^2,0). $
The definition of the invariants in this case is more delicate. The Euler obstruction of a map-germ $f: (X,0) \to (\mathbb C^k,0)$ was defined by Grulha Jr. in \cite{Grulha2008}. The main result
in that paper is a formula that computes the Euler obstruction of $f$  in terms of the Euler
obstruction for $k$-frames as defined by Brasselet, Seade and Suwa \cite{BrasseletSeadeSuwa2009}.
Moreover, Ebeling and Gusein-Zade introduced in \cite{chernobstructions} the notion of Chern
obstruction for singular spaces using collections of differential $1$-forms. This number is well
defined for any germ of reduced equidimensional complex analytic space, and the authors
provide in \cite{brasselet2010euler} a formula for the Chern obstruction in terms of intersection number. Our inspiration was \cite{grulha2022geometrical}, in which the authors show that the Euler obstruction of a map can be
written as a Chern obstruction of a convenient collection of forms.

Let $(X,0)$ be the germ at the origin of an equidimensional reduced analytic variety in $(\mathbb C^n,0)$ and $f=(f_1, f_2): (X,0) \to (\mathbb C^2,0)$  a map-
germ such that $f_2: (X,0) \to (\mathbb C,0)$ is a function possibly with noninsolated singularity at the origin and $f_1: (X,0) \to (\mathbb C,0)$
is tractable at the origin with respect to a good stratification $\mathcal V$ of $X \cap g_2^{-1}(0)$ relative to $g_2$ (see Definition \ref{definition tractable}). Let $d$ be the dimension of $X\cap f_2^{-1}(0).$
We prove a relation (Proposition \ref{propositioncentral}) between the Chern number of the $1$-form $\{df_1\}$ and the relative Bruce-Roberts number
of the function  $f_1$ in $X\cap f_2^{-1}(0),$ when the  relative logarithmic characteristic variety $LC(X\cap f_2^{-1}(0))^{-}$ is Cohen-Macaulay. This is the case when $X$ and $X\cap f_2^{-1}(0)$ are ICIS, and we can benefit from formula \cite[Theorem 2.2]{lima2022}
%\textcolor{red}{$$
%\mu_{BR}^{-}(f,X)=\mu(X\cap f^{-1}(0))+\mu(X,0)-\tau(X,0),
%$$ }%(dar refer\^encia \`a f\'ormula da B\'arbara)
to easily calculate the relative Bruce-Roberts number, obtaining Theorem \ref{theorem central coleção de uma forma}.

In the last section, we consider a $2$-dimensional ICIS $(X,0)\subset (\mathbb{C}^n,0)$. Using the algebraic characterization of the Chern number of a collection of forms,
%defined on an ICIS $X,$ given by Ebeling and Gusein-Zade in \cite{chernobstructions},
we give an alternative description  for the number of cusps $c(f\big|_{X})$ of a stabilization of an $\mathcal A$-finite map-germ $f=(f_1,f_2): (X,0) \to (\mathbb C^2,0).$  A formula for $c(f\big|_{X})$ was first given in (\cite[Proposition 3.3]{juanjobrunatomazella3}).

\begin{center}\textbf{Acknowledgements}\end{center}

The first author was supported by FAPESP, under grants 2022/08662-1 and 2023/04460-8. The work of the second author was partially supported by FAPESP Proc. 2019/21181-0 and CNPq Proc. 305695/2019-3. The third author was supported by FAPESP, grant 2022/06968-6.

\section{Topological invariants}

In this section, we recall  the definition of the local Euler obstruction and the Euler obstruction of a function.

Let $(X,0)\subset(\mathbb{C}^n,0)$ be an equidimensional reduced complex analytic germ of dimension $d$ in an open set $U\subset\mathbb{C}^n.$ Consider a complex analytic Whitney stratification $\{V_{\lambda}\}$ of $U$ adapted to $X$ such that $\{0\}$ is a stratum. We choose a small representative of $(X,0),$ denoted by $X,$ such that $0$ belongs to the closure of all strata. We write $X=\cup_{i=0}^{q} V_i,$ where $V_0=\{0\}$ and $V_q$ denotes the regular part $X_{reg}$ of $X.$ We suppose that $V_0,V_1,\ldots,V_{q-1}$ are connected and that the analytic sets $\overline{V_0},\overline{V_1},\ldots,\overline{V_q}$ are reduced. We write $d_i=dim(V_i), \ i\in\{1,\ldots,q\}.$ Note that $d_q=d.$

Let $G(d,n)$ be the Grassmannian manifold, the space of $d$-dimensional vector subspaces in $\mathbb{C}^n$, $x\in X_{reg}$ and consider the Gauss map $\phi: X_{reg}\rightarrow U\times G(d,n)$ given by $x\mapsto(x,T_x(X_{reg})).$ The closure of the image of the Gauss map $\phi$ in $U\times G(d,n)$, denoted by $\tilde{X}$, is called \textbf{Nash modification} of $X$. It is a complex analytic space endowed with an analytic projection map $\nu:\tilde{X}\rightarrow X.$ Consider the extension of the tautological bundle $\mathcal{T}$ over $U\times G(d,n).$ Since \linebreak$\tilde{X}\subset U\times G(d,n)$, we consider $\tilde{T}$ the restriction of $\mathcal{T}$ to $\tilde{X},$ called the \textbf{Nash bundle}, and $\pi:\tilde{T}\rightarrow\tilde{X}$ the projection of this bundle.

In this context, denoting by $\varphi$ the natural projection of $U\times G(d,n)$ at $U,$ we have the following diagram:

$$\xymatrix{
\tilde{T} \ar[d]_{\pi}\ar[r] & \mathcal{T}\ar[d] \\
\tilde{X}\ar[d]_{\nu}\ar[r] & U\times G(d,n)\ar[d]^{\varphi} \\
X\ar[r] & U\subseteq\mathbb{C}^n \\}  $$

Adding the stratum $U \setminus X$ we obtain a Whitney stratification $\mathcal{V}$ of $U$. Let us denote the restriction to $X$ of the tangent bundle of $U$ by $TU|_{X}$. We know that a stratified vector field $v$ on $X$ means a continuous section of $TU|_{X}$ such that if $x \in V_{i} \cap
 X$ then $v(x) \in T_{x}(V_{i})$. From Whitney's condition (a), one has the following lemma.

\begin{lemma} \cite{brasselet1981classes}.
	Every stratified vector field $v$, non-null on a subset $A \subset X$, has a canonical lifting to a non-null section $\tilde{v}$ of the Nash bundle $\widetilde {T}$ over $\nu^{-1}(A) \subset \widetilde{X}$.
\end{lemma}

Now consider a stratified radial vector field $v(x)$ in a neighborhood of $\left\{0\right\}$ in $X$, {\textit{i.e.}}, there is $\varepsilon_{0}$ such that for every $0<\varepsilon \leq \varepsilon_{0}$, $v(x)$ is pointing outwards the ball ${B}_{\varepsilon}$ over the boundary ${S}_{\varepsilon} := \partial{{B}_{\varepsilon}}$.

The following interpretation of the local Euler obstruction has been given by Brasselet and Schwartz in \cite{brasselet1981classes}. As said before, the original definition is presented in \cite{MacPherson}.

\begin{definition}
	Let $v$ be a radial vector field on $X \cap {S}_{\varepsilon}$ and $\tilde{v}$ the lifting of $v$ on $\nu^{-1}(X \cap {S}_{\varepsilon})$ to a section of the Nash bundle. The \textbf{local Euler obstruction} (or simply the \textbf{Euler obstruction}) ${\rm Eu}_{X}(0)$ is defined to be the obstruction to extending $\tilde{v}$ as a nowhere zero section of $\widetilde {T}$ over $\nu^{-1}(X \cap {B}_{\varepsilon})$.
\end{definition}

More precisely, let $\mathcal{O}{(\tilde{v})} \in \mathbb{H}^{2d}(\nu^{-1}(X \cap
{B_{\varepsilon}}), \nu^{-1}(X \cap {S_{\varepsilon}}), \mathbb{Z})$ be the class of the
obstruction cocycle to extending $\tilde{v}$ as a nowhere zero section of
$\widetilde {T}$ inside $\nu^{-1}(X\cap {B_{\varepsilon}})$. The
local Euler obstruction ${\rm Eu}_{X}(0)$ is defined as the evaluation of the
class $\mathcal{O}(\tilde{v})$ on the fundamental class of the pair $[\nu^{-1}(X
\cap {B_{\varepsilon}}), \nu^{-1}(X \cap {S_{\varepsilon}})]$ and therefore
it is an integer.

Let us give the definition of another invariant introduced by Brasselet, Massey, Parameswaran and Seade in \cite{BMPS}. Let $f:X\rightarrow\mathbb{C}$ be a holomorphic function with isolated singularity at the origin given by the restriction of a holomorphic function $F:U\rightarrow\mathbb{C}$ and denote by $\overline{\nabla}F(x)$ the conjugate of the gradient vector field of $F$ at $x\in U,$ $$\overline{\nabla}F(x):=\left(\overline{\frac{\partial F}{\partial x_1}},\ldots, \overline{\frac{\partial F}{\partial x_n}}\right).$$

Since $f$ has an isolated singularity at the origin, for all $x\in X\setminus\{0\},$ the projection $\hat{\zeta}_i(x)$ of $\overline{\nabla}F(x)$ over $T_x(V_i(x))$ is nonzero, where $V_i(x)$ is a stratum containing $x.$ Using this projection, Brasselet, Massey, Parameswaran and Seade constructed, in \cite{BMPS}, a stratified vector field on $X,$ denoted by $\overline{\nabla}f(x).$ Let $\tilde{\zeta}$ be the lifting of $\overline{\nabla}f(x)$ as a section of the Nash bundle $\tilde{T}$ over $\tilde{X}$, without singularity on $\nu^{-1}(X\cap S_{\varepsilon}).$

Let $\mathcal{O}(\tilde{\zeta})\in\mathbb{H}^{2n}(\nu^{-1}(X\cap B_{\varepsilon}),\nu^{-1}(X\cap S_{\varepsilon}), \mathbb{Z})$ be the class of the obstruction cocycle for extending $\tilde{\zeta}$ as a non zero section of $\tilde{T}$ inside $\nu^{-1}(X\cap B_{\varepsilon}).$

\begin{definition}
The \textbf{local Euler obstruction of the function} $f, Eu_{f,X}(0)$ is the evaluation of $\mathcal{O}(\tilde{\zeta})$ on the fundamental class $[\nu^{-1}(X\cap B_{\varepsilon}),\nu^{-1}(X\cap S_{\varepsilon})].$
\end{definition}

In \cite[Proposition 2.3 ]{STV}, Seade, Tib\u{a}r and Verjovsky proved that the Euler obstruction of a function $f$ is also related to the number of Morse critical points of a stratified Morsification of $f.$

\begin{proposition}\label{Eu_f and Morse points}
Let $f:(X,0)\rightarrow(\mathbb{C},0)$ be a germ of holomorphic function with isolated singularity at the origin. Then, \begin{center}
$Eu_{f,X}(0)=(-1)^dn_{reg},$
\end{center}
where $n_{reg}$ is the number of Morse points in $X_{reg}$ of a stratified Morsification of $f.$
\end{proposition}

Let $X$ be a reduced complex analytic space (not necessarily equidimensional) of dimension $d$ in an open set $U\subseteq\mathbb{C}^n$ and let $f:(X,0)\rightarrow(\mathbb{C},0)$ be a holomorphic function. The following definitions are due to Massey (\cite{Ms2}).

\begin{definition}\label{good stratification}
A \textbf{good stratification of $X$ relative to $f$} is a stratification $\mathcal{V}$ of $X$ which is adapted to $X\cap\{f=0\}$ such that $\{V_{\lambda}\in\mathcal{V},V_{\lambda}\nsubseteq X\cap\{f=0\}\}$ is a Whitney stratification of $X\setminus X\cap\{f=0\}$ and such that for any pair $(V_{\lambda},V_{\gamma})$ such that $V_{\lambda}\nsubseteq X\cap\{f=0\}$ and $V_{\gamma}\subseteq X\cap\{f=0\},$ the $(a_f)$-Thom condition is satisfied, that is, if $p\in V_{\gamma}$ and $p_i\in V_{\lambda}$ are such that $p_i\rightarrow p$ and $T_{p_i} V(f|_{V_{\lambda}}-f|_{V_{\lambda}}(p_i))$ converges to some $\mathcal{T},$ then $T_p V_{\gamma}\subseteq\mathcal{T}.$
\end{definition}

Let $\mathcal{V}$ be a good stratification of $X$ relative to $f$ and $g:(X,0)\rightarrow(\mathbb{C},0)$ a holomorphic function-germ.

\begin{definition}
If $\mathcal{V}=\{V_{\lambda}\}$ is a stratification of $X,$ the \textbf{symmetric relative polar variety of $f$ and $g$ with respect to $\mathcal{V}$}, $\tilde{\Gamma}_{f,g}(\mathcal{V}),$ is the union $\cup_{\lambda}\tilde{\Gamma}_{f,g}(V_{\lambda}),$ where $\tilde{\Gamma}_{f,g}(V_{\lambda})$ denotes the closure in $X$ of the critical locus of $(f,g)|_{V_{\lambda}\setminus (X^f\cup X^g)},$  $X^f=X\cap \{f=0\}$ and $X^g=X\cap \{g=0\}. $
\end{definition}

\begin{comment}
\begin{definition}
If $\mathcal{V}=\{V_{\lambda}\}$ is a stratification of $X,$ the \textbf{relative polar variety of $f$ and $g$ with respect to $\mathcal{V}$}, denoted by $\Gamma_{f,g}(\mathcal{V}),$ is the the union $\cup_{\lambda}\Gamma_{f,g}(V_{\lambda}),$ where $\Gamma_{f,g}(V_{\lambda})$ denotes the closure in $X$ of the critical locus of $(f,g)|_{V_{\lambda}\setminus X^f}.$
\end{definition}

\begin{definition}\label{definition prepolar}
Let $\mathcal{V}$ be a good stratification of $X$ relative to a function
$f:(X,0)\rightarrow(\mathbb{C},0).$ A function $g :(X, 0)\rightarrow(\mathbb{C},0)$ is \textbf{prepolar with respect to $\mathcal{V}$ at the origin} if the origin is a stratified isolated critical point, that is, $0$ is an isolated point of $\Sigma_{\mathcal{V}}g.$
\end{definition}

\end{comment}

\begin{definition}\label{definition tractable}
A function $g :(X, 0)\rightarrow(\mathbb{C},0)$ is \textbf{tractable at the origin with respect to a good stratification $\mathcal{V}$ of $X$ relative to $f :(X, 0)\rightarrow(\mathbb{C},0)$} if $dim_0 \ \tilde{\Gamma}_{f,g}(\mathcal{V})\leq1$ and, for all strata $V_{\alpha}\subseteq X\cap \{f=0\}$,
$g|_{V_{\alpha}}$ has no critical point in a neighborhood of the origin except perhaps at the origin itself.
\end{definition}

%\textcolor{red}{\noindent\textbf{Remark.} Tirar e adaptar todos os enunciados que usam isso. \label{afinito} A class of functions which satisfies the tractability condition is the following. Let $g=(g_1,g_2):(X,0)\rightarrow(\C^2,0)$ be an $\mathcal{A}$-finite map defined on an ICIS $(X,0)$ (see \cite[Definition 2.1]{juanjobrunatomazella3}) such that $X\cap g_2^{-1}(0)$ and $X\cap g_2^{-1}(0)\cap g_1^{-1}(0)$ are also ICIS. Hence, the critical set of the map $g$ is a reduced curve (see \cite{BOT}), that is, $\dim\tilde{\Gamma}_{g_1,g_2}(\mathcal{V})\leq1,$ where $\mathcal{V}=\{X\setminus g_2^{-1}(0), g_2^{-1}(0)\setminus\{0\},\{0\}\}$ is a good stratification of $X$ relative to $g_2.$ Moreover, if $g$ is an $\mathcal{A}$-finite map then $g$ is a finite map, which implies that $X\cap g_2^{-1}(0)\cap g_1^{-1}(0)$ reduces to the origin. Hence, $g_1|_{X\cap g_2^{-1}(0)}$ has at most isolated singularity at the origin. Therefore,  $g_1$ is tractable at the origin with respect to $\mathcal{V}$ relative to $g_2.$}

\vspace{0,5cm}

We justify the choice of the stratified approach in the following. In \cite{Ms1}, Massey characterize a certain number of Morse critical points in a result we partially enunciate above.

\begin{theorem}\cite[Theorem 3.2]{Ms1}
Let $X$ be an analytic subset of $\mathbb{C}^n$ with $0\in X$. Let $\mathcal{S}=\{S_{\alpha}\}$ be a Whitney stratification of $X$ and suppose that $f:(X,0)\rightarrow(\mathbb{C},0)$ has a stratified isolated critical point at $0$. There exists a unique set $\{k_{\alpha}\}$ of non-negative integers such that $k_{\alpha}$ is the number of non-degenerate critical points of a small perturbation of $f$ which occur near the origin on the stratum $S_{\alpha}$, more precisely, for all sufficiently small $\epsilon>0,$ for all complex $t$ such that $0<|t|<<\epsilon, k_{\alpha}$ equals the number of critical points of $f+tL$ in $int(B_{\epsilon})\cap S_{\alpha}$, for a generic choice of linear forms $L$.
\end{theorem}

\section{Chern number}

The notion of the Chern number, introduced in \cite{Ebeling2007a} extends the notion of local Euler obstruction to collections of $1$-forms. Let $(X, 0)\subset(\mathbb{C}^{n},0)$ be the germ of a $d$-equidimensional reduced complex analytic variety at the origin. Let $\{\omega_j^{(i)}\}$ be a collection of germs of 1-forms on $(\mathbb{C}^{n},0)$ such that $i=1,\ldots, s$; $j=1, \ldots, d-k_{i}+1$, where the $k_{i}$ are non-negative integers with $\sum_{i=1}^{s}k_{i}=d$. Let $\varepsilon>0$ be small enough so that there is a representative $X$ of the germ $(X, 0)$ and representatives $\{\omega_{j}^{(i)}\}$ of the germs of 1-forms inside the ball $B_{\varepsilon}(0)\subset\mathbb{C}^{n}.$ The following definitions are due to Ebeling and Gusein-Zade (\cite{Ebeling2007a}).

\begin{definition}
	For a fixed $i$, the locus of the subcollection $\omega_{j}^{(i)}$ is the set of points $x\in X$ such that there exists a sequence $x_{n}$ of points from the non-singular part $X_{{\rm{reg}}}$ of the variety $X$ such that the sequence $T_{x_{n}}X_{{\rm{reg}}}$ of the tangent spaces at the points $x_{n}$ has a limit $L$ (in $G(d,n)$) and the restrictions of the 1-forms $\omega_{1}^{(i)}, \ldots, \omega_{d-k_{i}+1}^{(i)}$ to the subspace $L\subset T_{x}\mathbb{C}^{n}$ are linearly dependent.
\end{definition}

\begin{definition}\label{specialpoint} A point $x\in X$ is called a {\it special point} of the collection $\{\omega_{j}^{(i)}\}$ if it is in the intersection of the loci of the subcollections $\omega_{j}^{(i)}$ for each $i=1,\ldots, s$. The collection $\{\omega_{j}^{(i)}\}$ of 1-forms has an isolated special point at $\{0\}$ if it has no special point on $X$ in a punctured neighborhood of the origin.
\end{definition}

Let $\{\omega_{j}^{(i)}\}$ be a collection of germs of 1-forms on $(X,0)$ with an isolated special point at the origin. Let $\nu:\tilde{X}\rightarrow X$ be the Nash transformation of the variety $X$ and $\tilde{T}$ be the Nash bundle. The collection of 1-forms $\{\omega_{j}^{(i)}\}$ defines a section $\Gamma(\omega)$ of the bundle
$$\tilde{\mathbb{T}}=\bigoplus_{i=1}^{s}\bigoplus_{j=1}^{d-k_{i}+1}\tilde{T}_{i,j}^{*},$$ where $\tilde{T}_{i,j}^{*}$ are copies of the dual Nash bundle $\tilde{T}^{*}$ over the Nash transformation $\tilde{X}.$

Let $\mathbb{D}\subset\tilde{\mathbb{T}}$ be the set of pairs $(x,\{\alpha_{j}^{(i)}\})$ where $x\in\tilde{X}$ and the collection of 1-forms $\{\alpha_{j}^{(i)}\}$ is such that $\alpha_{1}^{(i)},\ldots, \alpha_{n-k_{i}+1}^{(i)}$ are linearly dependent for each $i=1, \ldots, s.$

\begin{definition}
	Let $0$ be a special point of the collection $\{\omega_{j}^{(i)}\}$. The local Chern obstruction ${\rm{Ch}}_{X,0}\{\omega_{j}^{(i)}\}$ of the collection of germs of 1-forms $\{\omega_{j}^{(i)}\}$ on $(X,0)$ at the origin is the obstruction to extend the section $\Gamma(\omega)$ of the fibre bundle $\tilde{\mathbb{T}}\backslash\mathbb{D}\rightarrow \tilde{X}$ from $\nu^{-1}(X\cap S_{\varepsilon})$ to $\nu^{-1}(X\cap B_{\varepsilon})$.
\end{definition}

There exists an equivalent construction as the one made above choosing $k$-frames (copies of vector fields) instead of collection of $1$-forms. Such construction was made by Grulha Jr. in \cite{Grulha2008}. In that paper, the author uses the Euler obstruction of $k$-frames (\cite{BrasseletSeadeSuwa2009}) to define the Euler obstruction of a map-germ $f: (X,0) \to (\mathbb C^k,0)$, which is denoted by $Eu_{f,X}(0).$ In \cite[Theorem 3.2]{brasselet2010euler}, the authors prove that ${\rm Eu}_{f,X}(0)=(-1)^{d-1}{\rm{Ch}}_{X,0}\{\omega_{j}^{(i)}\}.$ In this sense, the Euler obstruction of a map and the Chern obstruction can be seen as a generalization of the Euler obstruction of function.

The following result is a consequence of \cite[Proposition 3.3]{Ebeling2007a}.

\begin{proposition}\rm{
Let $(X,0)\subset(\mathbb{C}^{n},0)$ be the germ of a $d$-equidimensional reduced complex analytic variety at the origin. Let $\{\omega_{j}^{(i)}\}$ be a collection of germs of $1$-forms on $(\mathbb{C}^{n},0)$ such that $i=1, \ldots, s$; $j=1, \ldots, d-k_{i}+1,$ where the $k_i$ are non-negative integers with $\sum_{i=1}^{s}k_{i}=d$. Let $0$ be an isolated special point for the collection. If $\omega^{(i)}$, $i=2, \ldots, s$, are generic collections of linear forms, then the number ${\rm{Ch}}_{X,0}\{\omega_{j}^{(i)}\}$ does not depend on the choice of the subcollections $\omega^{(i)}$, $i=2, \ldots, s$.}
\end{proposition}

For collection of $1$-forms it is also possible to define a notion of non-degenerate special points \cite{EGZLondon}. In \cite{Ebeling2007a}, the authors prove that one can deform a collection of $1$-forms with isolated special point into a collection of $1$-forms with only non-degenerate special points.

\begin{proposition}\cite[Corollary 1]{Ebeling2007a}
Let $\{\omega_{j}^{(i)}\}$ be a collection of $1$-forms on the $d$-equidimen-\ sional reduced complex analytic variety at the origin $X$. If $\{\omega_{j}^{(i)}\}$ has isolated special point at the origin, there exists a deformation $\{\tilde{\omega}_{j}^{(i)}\}$ of the collection $\{\omega_{j}^{(i)}\}$ whose special points lie in the regular part of $X$ and they are all non-degenerate. Moreover, as such a deformation we can use $\{\omega_{j}^{(i)}+\lambda l_{j}^{(i)}\}$, with a collection $\{l_{j}^{(i)}\}$ of generic linear forms, where $\lambda\neq0$ is small enough.
\end{proposition}

Notice that the deformation used by Massey in \cite[Theorem 3.2]{Ms1} and by Ebeling and Guzein-Zade in \cite[Corollary 1]{Ebeling2007a} have the same structure.

%In \cite{Grulha2008}, Grulha defined the notion of the Euler obstruction of a map. In \cite{brasselet2010euler}, the authors proved that the Euler obstruction of a map is up to sign equal to the Chern number, defined by Ebeling and Gusein-Zade \cite{Ebeling2007a}. Following \cite{grulha2022geometrical}, we define the Euler obstruction of a map in terms of collection of forms.

%\begin{definition}
	%Let $X$ be an equidimensional complex variety of dimension $d$, $f:(X,0)\to \mathbb{C}^{p}$, a holomorphic map, with $0\leq p \leq d$ and $\omega_{1}=\{df_{1}, df_{2},...,df_{p}\}$, with $df_{i}$ the differential of the coordinate functions of $f$, and $\omega_{2}$ a generic collection, in such way that $0$ is a special point of the collection of collections $\omega = \{\omega_{1},\omega_{2}\}$. We define the Euler obstruction of the map $f$ at the origin, denoted by $Eu
%^{*}_{f,X}(0)={\rm{Ch}}_{X,0}\{\omega\}$.

%\end{definition}

\section{Relative Bruce-Roberts number}
Let $(X,0)\subset(\C^{n},0)$ be a reduced analytic variety and $I_{X}$ the ideal in the ring of analytic function germs $\mathcal{O}_{n}$ that vanish on $(X,0)$. We denote by $\Theta_{X}$ the $\mathcal{O}_{n}$-module of germs of vector fields that are tangent to $(X,0)$, it is,
$$
\Theta_{X}=\{\xi\in\Theta_{n};\;dh(\xi)\in I_{X}\;\forall h\in I_{X}\},
$$
where $\Theta_{n}$ is the $\mathcal{O}_{n}$-module of germs of vector fields on $(\C^{n},0).$ Let $f:(\C^{n},0)\to(\C,0)$ be a function germ in $\mathcal{O}_{n}$. The relative Bruce-Roberts number of $f$ with respect to $(X,0)$, $\mu_{BR}^{-}(f,X)$, is defined by Bruce and Roberts in \cite{bruce1988critical} by
$$
\mu_{BR}^{-}(f,X)=\dim_{\C}\frac{\mathcal{O}_{n}}{df(\Theta_{X})+I_{X}}.
$$
An interesting fact about this number is that it can be finite even when $f$ has no isolated singularity. %Actually this number is finite when $(\Sigma_{f}\cap X,0)\subset(\{0\},0),$ where $\Sigma_{f}$ is the singular set of $f$}.
Moreover as noted in \cite{bruce1988critical}, this number can be interpreted as the codimension of the orbit of the action of $\mathcal{R}_{X}$ in $\mathcal{O}_{X}$ given by
\begin{align*}
\overline{\Psi}:\mathcal{R}_{X}\times\mathcal{O}_{X}&\to\mathcal{O}_{X}\\
(h,f)&\mapsto f\circ h,
\end{align*}
where $\mathcal{R}_{X}$ is the group of germs of diffeomorphisms in $(\C^{n},0)$ preserving $(X,0)$.

Another property of the relative Bruce-Roberts number is its relation with the relative logarithmic characteristic variety. We assume that
$$
\Theta_{X}=\langle\delta_{1},...,\delta_{m}\rangle,\textup{ in which }\;\delta_{j}=\sum_{i=1}^{n}\delta_{ij}e_{i},\;\forall j=1,...,m
$$
and let $x_{1},...,x_{n},\xi_{1},...,\xi_{n}$ be the coordinates of the cotangent bundle $T^{*}\C^{n}.$

%Suppose the vector fields $\delta_1,\dots ,\delta_m$ generate $\Theta_X$ on some neighborhood $U$ of $0$ in $\C^n$. Let $T^*_U\C^n$ be the restriction of the cotangent bundle of $\C^n$ to $U$. We define $LC_U(X)$ to be
Let $U$ be a neighborhood of zero in $\C^{n}$. The logarithmic characteristic variety of $(X,0)$, $LC(X)$, is the germ of
$$
LC_U(X)=\{(x,\xi)\in T^*_U\C^n:\xi(\delta_i(x))=0, i=1,\dots ,m\},
$$
along $T^*_0\C^n$, the cotangent space to $\C^n$ at $0$.
It does not depend of the choice of the generators of $\Theta_{X}$, and  its algebraic structure is given by the ideal
$$
I=\left\langle\sum_{i=1}^{n}\delta_{ij}\xi_{i},\;j=1,...,m\right\rangle,
$$
 then it is not necessarily reduced.

Let $\{X_{\alpha}\}$ be the logarithmic stratification of $U$, it is, $\{X_{\alpha}\}$ is the unique stratification such that
\begin{itemize}
\item [(i)]\ Each stratum $X_{\alpha}$ is a smooth connected immersed submanifold of $U$ and $U$ is the disjoint union  $\cup_{\alpha\in I}X_{\alpha}.$
\item[(ii)] If $x\in U$ lies in a stratum $X_{\alpha}$, then the tangent space $T_{x}X_{\alpha}$ coincides with $\Theta_{X}(x).$
\item [(iii)]If $X_\alpha$ and $X_\beta$ are two distinct strata with $X_\alpha$ meeting the closure of $X_\beta$ then $X_\alpha$ is contained in the frontier of $X_\beta$.
\end{itemize}
Moreover, the connected components of $U\setminus X$ are logarithmic strata and, when $(X,0)$ is equidimensional, the connected components of $X\setminus \Sigma_{X}$ are logarithmic strata, where $\Sigma_{X}$ is the singular set of $(X,0)$. See \cite{bruce1988critical, saito1980theory} for more details.

 If for some neighborhood of the origin the logarithmic stratification has a finite number of strata the variety is called holonomic. Let $\{X_{\alpha}\}$ be the logarithmic stratification of $(X,0)$ then as a set germs
$$
LC(X)=\bigcup_{\alpha=0}^{s}Y_{\alpha}=\bigcup_{\alpha=0}^{s}\overline{N^{*}X_{\alpha}},
$$
where $Y_{\alpha}=\overline{N^{*}X_{\alpha}}$ is the closure of the conormal bundle $N^{*}X_{\alpha}$ of $X_{\alpha}$ in $T^{*}\C^{n}.$ If $(X,0)$ is holonomic then $Y_{\alpha}=\overline{N^{*}X_{\alpha}}$ are the irreducible components of $LC(X)$, see \cite[Proposition 1.14(ii)]{bruce1988critical}, and we denote the geometrical multiplicity of $Y_{\alpha}$ in $LC(X)$ by $m_{\alpha}$. Denoting the logarithmic stratum $U\setminus X$ by $X_{0}$,  the relative logarithimic characteristic variety of $(X,0)$ is defined in \cite[p.72]{bruce1988critical} as the closure of the subset of $LC(X)$ obtained by deleting the component $Y_{0}$ coming from the $X_{0}$. Then as set germs
$$
LC(X)^{-}=\bigcup_{\alpha=1}^{s}\overline{N^{*}X_{s}}.
$$
When $(X,0)$ is a holonomic variety there exists a geometrical characterization for the relative Bruce-Roberts number.

\begin{proposition}\cite[Proposition 5.11]{bruce1988critical}\label{Proposition5.11}
Let $\{X_{\alpha}\}_{\alpha=0}^{s}$ be the logarithmic stratification of $(X,0)$ with $X_{0}=U\setminus X$. If $f$ has an isolated critical points at $0$ and $n_{\alpha}$ is the number of critical points of a Morsification of $f$ on $X_{\alpha}$, then
$$
\sum_{i=1}^{s}n_{i}m_{i}\leq \mu_{BR}^{-}(f,X),
$$
where $m_{i}$ is the multiplicity of $Y_{i}$ in $LC(X)$. The quality holds if and only if $LC(X)^{-}$ is Cohen-Macaulay at $(0,df(0))$.
\end{proposition}

The previous characterization is an important tool in order to prove that $LC(X)^{-}$ is Cohen-Macaulay and it is fundamental in \cite{lima2022} where the following result is given.% they prove the following result.

\begin{theorem}\cite[Theorem 3.1]{lima2022}
If $(X,0)\subset(\C^{n},0)$ is an ICIS, then  $LC(X)^{-}$ is Cohen-Macaulay. %for any ICIS  $(X,0) \subset (\C^{n},0)$.
\end{theorem}

 This theorem extends \cite[Proposition 3.2 ]{lima2021relative} and \cite[Proposition 7.3 (i)]{bruce1988critical} in which they consider any isolated hypersuface singularity (IHS) and weighted homogeneous ICIS, respectively.

\section{The main result}

Our goal is to study invariants of map germs   $f=(f_{1}, f_{2}):(X,0)\subset (\C^{n},0) \to \C^{2}$ whose singular set is a one-dimensional reduced curve or empty.  We want to compare the previously defined invariants giving an algebraic meaning for the Chern number.

In \cite[Corollary 3.1]{brasselet2010euler} the authors construct a collection of $1$-forms for which one has a equality between the Euler obstruction of a map and the Chern number.

\begin{proposition}\cite[Corollary 3.1]{brasselet2010euler}
Let $(X,0)$  be a reduced analytic variety and $f:(X,0)\rightarrow(\mathbb{C}^p,0)$ a map-germ defined on $X$. There exists a collection $\{\omega^{(i)}_j\}$ such that $Eu^{*}_{f,X}(0)=Ch_{X,0}\{\omega^{(i)}_j\}.$
\end{proposition}

In \cite{grulha2022geometrical}, the authors investigate the last formula in the case of a map $f=(f_1,f_2):(X,0)\subset(\mathbb{C}^n,0)\rightarrow(\mathbb{C}^2,0)$, where $f_1$ is tractable at the origin with respect to a good stratification of $X\cap f_2^{-1}(\delta)$ relative to $f_2$, for a regular value $\delta$ of $f_2$. We focus our analysis on a map $f=(f_1,f_2):(X,0)\rightarrow(\mathbb{C}^2,0)$ where $f_1$ is tractable at the origin with respect to a good stratification of $X\cap f_2^{-1}(0)$ relative to $f_2$ and $f_2:(X,0)\rightarrow(\mathbb{C},0)$ is a function possibly with nonisolated singularity at the origin. Following the construction made in \cite{brasselet2010euler}, we may consider the collection of $1$-forms $\big\{\{df_1,df_2\},\{l_1,\ldots,l_d\}\big\}$, where $\{l_1,\ldots,l_d\}$ is a collection of generic linear forms and $d$ is the dimension of $X\cap f_2^{-1}(0).$ This collection restricted to $X\cap f_2^{-1}(0)$ reduces to the $1$-form $\{df_1\}$. In fact, consider a point $p\in X\cap f^{-1}_2(0)$ and a sequence of points $(p_i)\in (X\cap f^{-1}_2(0))_{reg}$ converging to $p$. Let  $L\in G(d,n)$ be the limit of the sequence of tangent spaces $T_{p_i} (X\cap f^{-1}_2(0)_{reg}),$ then $ df_2|_{L}$ vanishes on $X\cap f^{-1}_2(0)$. We conclude that we may study the behavior of the $1$-form $\{df_1\}$ on $X\cap f^{-1}_2(0)$. We begin using the Chern number to count a certain number of Morse critical points.

%Let $(X,0)\subset(\mathbb{C}^n,0)$ be a reduced analytic variety and $g=(g_1,g_2):(X,0)\rightarrow(\mathbb{C}^2,0)$ a map such that $g_2:(X,0)\rightarrow(\mathbb{C},0)$ is a function possibly with nonisolated singularity at the origin and $g_1:(X,0)\rightarrow(\mathbb{C},0)$ is tractable at the origin with respect to a good stratification $\mathcal{V}$ of $X\cap g_2^{-1}(0)$ relative to $g_2$. %Consider the collection of $1$-forms $\{\omega^{(i)}_j\}=\big\{\{dg_1,dg_2\},\{l_1,\ldots,l_d\}\big\}$, where $\{l_1,\ldots,l_d\}$ is a collection of generic linear forms and $d$ is the dimension of $X\cap g_2^{-1}(0).$ The next proposition presents an extension of Theorem 4.1 in \cite{grulha2022geometrical} to this more general setting.

\begin{proposition}\label{lemacentralnonisolatedcase}
Let $(X,0)\subset(\mathbb{C}^n,0)$ be a reduced analytic variety and $f_2:(X,0)\rightarrow(\mathbb{C},0)$ be a function possibly with nonisolated singularity at origin. Suppose that $f_1:(X,0)\rightarrow(\mathbb{C},0)$ is tractable at the origin with respect to a good stratification $\mathcal{V}$ of $X\cap f_2^{-1}(0)$ relative to $f_2$. Then the Chern number $Ch_{X\cap f_2^{-1}(0),0}\{df_1\}$ is equal to the number of nondegenerated Morse critical points of a Morsification of $f_1$ appearing in the regular part of $X\cap f_2^{-1}(0).$%, where $\{\omega^{(i)}_{j}\}=\big\{\{dg_1,dg_2\},\{l_1,\ldots,l_{d}\}\big\}$, $l_j$ are generic linear forms and $d$ is the dimension of $X\cap g_2^{-1}(0).$
\end{proposition}
\noindent\textbf{Proof.} Let $\tilde{f_1}\big|_{V_{\alpha}}$ be a Morsification of $f_1\big|_{V_{\alpha}}$ in a stratum $V_{\alpha}\in\mathcal{V}$. Since  $f_1$ is tractable at the origin with respect to the good stratification $\mathcal{V}$ of $X\cap f_2^{-1}(0)$ relative to $f_2$, $f_1\Big|_{V_{\alpha}}$ has at most isolated singularity at the origin. Then, by \cite[Theorem 3.2]{Ms1}, $\tilde{f_1}\big|_{V_{\alpha}}=f_1\big|_{V_{\alpha}}+tH\big|_{V_{\alpha}}$ for a complex number $t, |t|<<1$ and a generic linear form $H.$ If $p$ is a nondegenerated critical point of $\tilde{f_1}\big|_{V_{\alpha}}$, \begin{center}
  $d\tilde{f}_1\big|_{T_p(V_{\alpha})}=df_1\big|_{T_p(V_{\alpha})}+tH\big|_{V_{\alpha}}=0.$
\end{center}

Hence, $p$ is in the locus of a generic perturbation (see \cite[Corollary 1]{chernobstructions}) of the $1$-form $\{df_1\big|_{T_p(V_{\alpha})}\}$, that is, $p$ is a special point of $\{df_1\}.$ By \cite[Proposition 2 ]{chernobstructions}, $p$ is counted by the Chern number $Ch_{X\cap f_2^{-1}(0),0} \{df_1\}.$
Conversely, let $p$ be a nondegenerated special point of a deformation $\{df_1+tH\}$ of the $1$-form $\{df_1\},$ where $t, |t|<<1,$ is a complex number and $H$ is a generic linear form.  Then, for a stratum $V_{\alpha}\in\mathcal{V}$, $\{df_1\big|_{T_p(V_{\alpha})}+tH\}$ is linearly dependent, that is, $df_1\big|_{T_p(V_{\alpha})}+tH=0.$  Hence, $p$ is a critical point of $df_1\big|_{T_p(V_{\alpha})}+tH,$ which is a Morsification of $f_1\big|_{T_p(V_{\alpha})}$ (see \cite[Theorem 3.2]{Ms1}).
\qed

The relation between a certain number of Morse critical points of a function and the number of nondegenerated special points of a collection of $1$-forms implies the following.

\begin{corollary}\label{corolariocentral1}
Let $(X,0)\subset(\mathbb{C}^n,0)$ be a reduced analytic variety and $f=(f_1,f_2):(X,0)\rightarrow(\mathbb{C}^2,0)$ be a map such that $f_2:(X,0)\rightarrow(\mathbb{C},0)$ is a function possibly with nonisolated singularity at the origin and $f_1:(X,0)\rightarrow(\mathbb{C},0)$ is tractable at the origin with respect to a good stratification $\mathcal{V}$ of $X\cap f_2^{-1}(0)$ relative to $f_2$. Hence
\begin{center}
  $Ch_{X\cap f_2^{-1}(0),0}\{df_1\}=Eu_{f_1,X\cap f_2^{-1}(0)}(0).$
\end{center}
\end{corollary}
\noindent\textbf{Proof.} The number of nondegenerated Morse critical points of a Morsification of $f_1$ appearing in the regular part of $X\cap f_2^{-1}(0)$ is equal to $Eu_{f_1,X\cap f_2^{-1}(0)}(0)$, by \cite[Proposition 2.3]{STV}.\qed

A consequence of Corollary \ref{corolariocentral1} is the following proposition.

\begin{proposition}\label{prop}
Let $(X,0)\subset(\mathbb{C}^n,0)$ be an ICIS and $f=(f_1,f_2):(X,0)\rightarrow(\mathbb{C}^2,0)$ a map such that $X\cap f_2^{-1}(0)$ and $X\cap f_2^{-1}(0)\cap f_1^{-1}(0)$ are also ICIS and $f_1:(X,0)\rightarrow(\mathbb{C},0)$ is tractable at the origin with respect to a good stratification $\mathcal{V}$ of $X\cap f_2^{-1}(0)$ relative to $f_2$. For the $1$-form $\{df_1\}$,
  \begin{eqnarray*}
    Ch_{X\cap f_2^{-1}(0),0}\{df_1\}-Ch_{X\cap f_2^{-1}(\alpha),x_0}\{df_1\}=Eu_{X\cap f_2^{-1}(0)}(0)-1
    \end{eqnarray*}
\end{proposition}
\noindent\textbf{Proof.}
We use \cite[Theorem 3.1]{BMPS} to compute the Euler obstruction of $f_1$ on $X\cap f_2^{-1}(0)$. If $\{\{0\}, V_1,\ldots, V_q\}$ is a Whitney stratification of $X$, since $f_2$ has isolated singularity at the origin, $\{\{0\}, V_1\cap f_2^{-1}(0),\ldots, V_q\cap f_2^{-1}(0)\}$ is a  Whitney stratification of $X\cap f_2^{-1}(0).$ %Hence,
 % \begin{eqnarray*}
                               % \nonumber % Remove numbering (before each equation)
                                  %Eu_{g_1,X\cap g_2^{-1}(0)}(0)
                                 % &=& \sum_{i=1}^{q}\chi(V_i\cap g_2^{-1}(0)\cap g_1^{-1}(\delta)\cap B_{\epsilon})Eu_{X\cap g_2^{-1}(0)}(V_i\cap g_2^{-1}(0))\\
                               %\end{eqnarray*}
Also since $f_2$ has isolated singularity at the origin, $f_2^{-1}(0)$ intersects $X$ transversely out of the origin, which implies that $Eu_{X\cap f_2^{-1}(0)}(V_i\cap f_2^{-1}(0))=Eu_X(V_i)$. Therefore, by Corollary \ref{corolariocentral1}, \begin{comment} \begin{eqnarray*}
                               % \nonumber % Remove numbering (before each equation)
                                  Ch_{X\cap g_2^{-1}(0),0}\{dg_1\}
                                  &=& Eu_{g_1,X\cap g_2^{-1}(0)}(0)\\
                                  &=& Eu_{X\cap g_2^{-1}(0)}(0)-\sum_{i=1}^{q}\chi(V_i\cap g_2^{-1}(0)\cap g_1^{-1}(\delta)\cap B_{\epsilon})Eu_{X\cap g_2^{-1}(0)}(V_i\cap g_2^{-1}(0))\\
                                  &=&Eu_{X\cap g_2^{-1}(0)}(0)-\sum_{i=1}^{q}\chi(V_i\cap g_2^{-1}(\alpha)\cap g_1^{-1}(\delta)\cap B_{\epsilon})Eu_{X}(V_i)\\
                                  &=&Eu_{X\cap g_2^{-1}(0)}(0)- \sum_{i=1}^{q}\chi(V_i\cap g_2^{-1}(\alpha)\cap g_1^{-1}(\delta)\cap B_{\epsilon})Eu_{X\cap g_2^{-1}(\alpha)}(V_i\cap g_2^{-1}(\alpha))\\
                                  &=& Eu_{X\cap g_2^{-1}(0)}(0)-(Eu_{X\cap g_2^{-1}(\alpha)}(x_0)-Eu_{g_1,X\cap g_2^{-1}(\alpha)}(x_0))\\
                                  &=& Eu_{X\cap g_2^{-1}(0)}(0)-1+Ch_{X\cap g_2^{-1}(\alpha),x_0}\{dg_1\},
                               \end{eqnarray*}\end{comment}
                               \begin{eqnarray*}
                               % \nonumber % Remove numbering (before each equation)
                                  &Ch_{X\cap f_2^{-1}(0),0}\{df_1\}= Eu_{f_1,X\cap f_2^{-1}(0)}(0)=\hspace{6.6cm}\\
                                  &= Eu_{X\cap f_2^{-1}(0)}(0)-\sum_{i=1}^{q}\chi(V_i\cap f_2^{-1}(0)\cap f_1^{-1}(\delta)\cap B_{\epsilon})Eu_{X\cap f_2^{-1}(0)}(V_i\cap f_2^{-1}(0))\hspace{0.6cm}\\
                                  &=Eu_{X\cap f_2^{-1}(0)}(0)-\sum_{i=1}^{q}\chi(V_i\cap f_2^{-1}(\alpha)\cap f_1^{-1}(\delta)\cap B_{\epsilon})Eu_{X}(V_i)\hspace{3.1cm}\\
                                  &=Eu_{X\cap f_2^{-1}(0)}(0)- \sum_{i=1}^{q}\chi(V_i\cap f_2^{-1}(\alpha)\cap f_1^{-1}(\delta)\cap B_{\epsilon})Eu_{X\cap f_2^{-1}(\alpha)}(V_i\cap f_2^{-1}(\alpha))\hspace{0.4cm}\\
                                  &= Eu_{X\cap f_2^{-1}(0)}(0)-(Eu_{X\cap f_2^{-1}(\alpha)}(x_0)-Eu_{f_1,X\cap f_2^{-1}(\alpha)}(x_0))\hspace{4.2cm}\\
                                  &= Eu_{X\cap f_2^{-1}(0)}(0)-1+Ch_{X\cap f_2^{-1}(\alpha),x_0}\{df_1\},\hspace{6.4cm}
                               \end{eqnarray*}

\noindent where $x_0\in X\cap f_2^{-1}(\alpha).$ Hence,

    \begin{center}
     $Ch_{X\cap f_2^{-1}(0),0}\{df_1\}-Ch_{X\cap f_2^{-1}(\alpha),x_0}\{df_1\}=Eu_{X\cap f_2^{-1}(0)}(0)-1.$
    \end{center}\qed

The elements in Corollary \ref{corolariocentral1} are useful to describe the topological behavior of a map or a function defined on a singular variety. Nevertheless, these numbers are difficult to compute which justifies the search for characterizations that could present a less complicate computation. The next theorem shows a relation between the Chern number of the $1$-form $\{df_1\}$ and the relative Bruce-Roberts number of the function $f_1$ in $X\cap f_2^{-1}(0)$.

Let $U$ be a neighborhood of the origin in $\mathbb{C}^n$ and let us denote the logarithmic strata of $X\cap f_2^{-1}(0)$ by $X_0,X_1,\ldots,X_{k+1}$, where $X_0=U\setminus (X\cap f_2^{-1}(0)), X_1, \ldots, X_s$ denote the connected components of $(X\cap f_2^{-1}(0))_{reg}$ and $X_{s+1},\ldots,X_{k+1}$ denote the remaining strata. If the relative logarithmic characteristic variety $LC(X\cap f_2^{-1}(0))^{-}$ of $X\cap f_2^{-1}(0)$ is Cohen-Macaulay, by Proposition \ref{Proposition5.11}, we have that \begin{center}
  $\mu_{BR}^{-}(f_1,X\cap f_2^{-1}(0))=\sum_{\alpha=1}^{k+1}n_{\alpha}m_{\alpha},$
\end{center}
\noindent where $n_{\alpha}$ denotes the number of Morse critical points of a Morsification of $f_1\big|_{X\cap f_2^{-1}(0)}$ appearing in the stratum $X_{\alpha}$ and $m_{\alpha}$ denotes the geometrical multiplicity of the correspondent irreducible component of the logarithmic characteristic subvariety in $LC(X\cap f_2^{-1}(0))^{-}$.

\begin{proposition}\label{propositioncentral}
Let $(X,0)\subset(\mathbb{C}^n,0)$ be a reduced analytic variety and $f=(f_1,f_2):(X,0)\rightarrow(\mathbb{C}^2,0)$ be a map such that $f_1:(X,0)\rightarrow(\mathbb{C},0)$ is tractable at the origin with respect to a good stratification $\mathcal{V}$ of $X\cap f_2^{-1}(0)$ relative to $f_2$. Suppose that $X\cap f_2^{-1}(0)$ is a holonomic variety and that the relative logarithmic characteristic variety $LC(X\cap f_2^{-1}(0))^{-}$ of $X\cap f_2^{-1}(0)$ is Cohen-Macaulay. For the $1$-form $\{df_1\}$,

\begin{center}
     $Ch_{X\cap f_2^{-1}(0),0}\{df_1\}=\mu_{BR}^{-}(f_1,X\cap f_2^{-1}(0))-\sum_{\alpha=s+1}^{k+1}n_{\alpha}m_{\alpha}.$
    \end{center}

\end{proposition}
\noindent\textbf{Proof.} Since $f_1$ is tractable at the origin with respect to $\mathcal{V}$ relative to $f_2,$ then $f_1|_{X\cap f_2^{-1}(0)}$ has stratified isolated singularity at the origin which implies that the relative Bruce-Roberts number $\mu_{BR}^{-}(f_1,X\cap f_2^{-1}(0))$ is finite.  Now, if $X_1,\ldots, X_s$ are the connected components of $(X\cap f_2^{-1}(0))_{reg}$, then $m_{\alpha}=1$ and by Proposition \ref{lemacentralnonisolatedcase},  $Ch_{X\cap f_2^{-1}(0),0}\{df_1\}=n_{1}+\cdots+n_s.$ %By \cite[Proposition 1]{chernobstructions}, it is possible to consider a collection of 1-forms $\{h^{(i)}_{j}\}$ sufficiently generic in order to concentrate all the nonzero nondegenerated special points of the deformation $\{\omega^{i}_{j}+h^{i}_{j}\}$ on $(X\cap g_2^{-1}(0))_{reg}$. Hence $n_{s+1}=\cdots=n_{k}=0.$

Now since $LC(X\cap f_2^{-1}(0))^{-}$ is Cohen-Macaulay,
                               \begin{eqnarray*}
                               % \nonumber % Remove numbering (before each equation)
                                  \mu_{BR}^{-}(f_1,X\cap f_2^{-1}(0)) &=& \sum_{\alpha=1}^{k+1}n_{\alpha}m_{\alpha} \\
                                  %&=&\sum_{\alpha=1}^{s}n_{\alpha}m_{\alpha}+\sum_{\alpha=s+1}^{k+1}n_{\alpha}m_{\alpha}\\
                                  &=&\sum_{\alpha=1}^{s}n_{\alpha}+\sum_{\alpha=s+1}^{k+1}n_{\alpha}m_{\alpha}\\
                                  &=& Ch_{X\cap f_2^{-1}(0),0}\{df_1\}+\sum_{\alpha=s+1}^{k+1}n_{\alpha}m_{\alpha}.\qed
                                  \end{eqnarray*}
\begin{comment}
By Proposition \ref{propositioncentral}, \begin{eqnarray*}
                               % \nonumber % Remove numbering (before each equation)
                                  Ch_{X\cap g_2^{-1}(0),0}\{\omega^{(i)}_j\}= Eu_{X\cap g_2^{-1}(0)}(0)-1+Ch_{X\cap g_2^{-1}(\alpha),x_0}\{\omega^{(i)}_j\},
                               \end{eqnarray*}

\noindent where $x_0\in X\cap g_2^{-1}(\alpha).$ Hence,

    \begin{eqnarray*}
     Ch_{X\cap g_2^{-1}(\alpha),x_0}\{\omega^{(i)}_j\}=1+\mu_{BR}^{-}(g_1,X\cap g_2^{-1}(0))-Eu_{X\cap g_2^{-1}(0)}(0)-\sum_{\alpha=s+1}^{k+1}n_{\alpha}m_{\alpha}.\qed
    \end{eqnarray*}
\end{comment}

%\subsection{The special case of ICIS}
A corollary of Proposition \ref{propositioncentral} occurs if we suppose that $(X,0)$ is an ICIS, once in \cite{lima2021relative} the authors proved that the relative logarithmic characteristic variety of an ICIS is Cohen-Macaulay. For this result, let us present a characterization given by Ebeling and Guzein-Zade in \cite{Ebeling2007a} for the Chern number on ICIS.

Let $(X,0)$ be an ICIS determined  by $\phi=(\phi_{1},...,\phi_{k}):(\C^{n},0)\to(\C^{k},0)$ and $\{\omega_{j}^{(i)}\}$, with $i=1,...,s;\;j=1,...,(n-k)-k_{i}+1;\;\sum_{i=1}^{s}k_{i}=n-k$ a collection of $1$-forms in a neighborhood at the origin in $\C^{n}$ with no special points in $X\setminus\{0\}$. Let $I_{n-k_{i}+1} \begin{pmatrix}
d\phi\\ d\omega^{i}
\end{pmatrix}=J(\phi,\omega^{i})$ be the ideal generated by the minors of maximum order of the jacobian matrix of $(\phi,\omega^{i})=(\phi_{1},...,\phi_{k},\omega_{1}^{i},...,\omega_{n-k-k_{i}+1}^{i}), \forall i=1,...,s.$ Denote by $I_{X,\{\omega_{j}^{(i)}\}}$ the ideal $$ \langle\phi_{1},...,\phi_{k}\rangle+I_{n-k_{1}+1}\begin{pmatrix}
d\phi\\ d\omega^{1}
\end{pmatrix}+...+I_{n-k_{s}+1} \begin{pmatrix}
d\phi\\ d\omega^{s}
\end{pmatrix}.$$ With this notation, an algebraic definition (\cite{chernobstructions}) for the index $ind_{X,0}\{\omega_{j}^{(i)}\}$ is the following quocient   $$ind_{X,0}\{\omega_{j}^{(i)}\}=\dim_{\C}\frac{\mathcal{O}_{n}}{I_{X,\{\omega_{j}^{(i)}\}}}.$$  By \cite[Corollary 3]{chernobstructions}, we have an algebraic formula for the Chern number as follows: \begin{equation}\label{relação para o chern de uma ICIS}
Ch_{X,0}\{\omega_{j}^{(i)}\}=ind_{X,0}\{\omega_{j}^{(i)}\}-ind_{X,0}\{l_{j}^{(i)}\},
\end{equation}
\noindent where  $\{l_{j}^{(i)}\}$ is a collection of generic linear functions on $\C^{n}$.
%Still considering $(X,0)$ the previous ICIS determined by map germ $\phi=(\phi_{1},...,\phi_{n-2}):(\C^{n},0)\to(\C^{n-2},0)$, and

\begin{theorem}\label{theorem central coleção de uma forma}
Let $(X,0)\subset(\mathbb{C}^n,0)$ be a $(d+1)$-dimensional ICIS and $f=(f_1,f_2):(X,0)\rightarrow(\mathbb{C}^2,0)$ a map such that $X\cap f_2^{-1}(0)$ and $X\cap f_2^{-1}(0)\cap f_1^{-1}(0)$ are also ICIS and $f_1:(X,0)\rightarrow(\mathbb{C},0)$ is tractable at the origin with respect to a good stratification $\mathcal{V}$ of $X\cap f_2^{-1}(0)$ relative to $f_2$. Consider the $1$-form $\{df_1\}$. Then,

\begin{center}
  $Ch_{X\cap f_2^{-1}(0),0}\{df_1\}=\mu_{BR}^{-}(f_1,X\cap f_2^{-1}(0)) - \mu_{BR}^{-}(L,X\cap f_2^{-1}(0)).$
\end{center}
\end{theorem}

\noindent\textbf{Proof.} Since $X\cap f_{2}^{-1}(0)$ is an ICIS, it follows from Proposition\cite[Proposition 3.1]{lima2022} and Proposition\ref{Proposition5.11} that $\mu_{BR}^{-}(f_1,X\cap f_2^{-1}(0))=\sum_{\alpha=1}^{k+1}n_{\alpha}m_{\alpha},$ where $n_{\alpha}$ denotes the number of Morse critical points of a Morsification of $f_1\big|_{X\cap f_2^{-1}(0)}$ appearing in the stratum $X_{\alpha}$ and $m_{\alpha}$ denotes the geometrical multiplicity of the correspondent irreducible component of the relative logarithmic characteristic subvariety in $LC(X\cap f_2^{-1}(0))^{-}$. Now, if $X_1,\ldots, X_k$ are the connected components of $(X\cap f_2^{-1}(0))_{reg}$, then $m_{\alpha}=1$ and by Proposition \ref{lemacentralnonisolatedcase},  $Ch_{X\cap f_2^{-1}(0),0}\{df_1\}=n_{1}+\cdots+n_k.$ On the other hand, by \cite[Corollary 3]{chernobstructions}, \begin{center}
                                         $Ch_{X\cap f_2^{-1}(0),0}\{df_1\}= ind_{X\cap f_2^{-1}(0),0}\{df_1\}-ind_{X\cap f_2^{-1}(0),0}\{L\},$
                                       \end{center}
\noindent where $L$ is a generic linear form. By \cite[Theorem 4]{EGZLondon} and the Lê-Greuel formula for the Milnor number, \begin{center}
                                                                                $ind_{X\cap f_2^{-1}(0),0}\{df_1\}=\mu(X\cap f_2^{-1}(0),0)+\mu(X\cap f_2^{-1}(0)\cap f_1^{-1}(0),0)$
                                                                     \end{center}
and
                                                                     \begin{center}
                                                                     $ind_{X\cap f_2^{-1}(0),0}\{L\}=\mu(X\cap f_2^{-1}(0),0)+\mu(X\cap f_2^{-1}(0)\cap L^{-1}(0),0).$
                                                                     \end{center}
By \cite[Theorem 2.2]{lima2021relative}, \begin{center}
                                               $\mu_{BR}^{-}(f_1,X\cap f_2^{-1}(0))=\mu(X\cap f_2^{-1}(0)\cap f_1^{-1}(0),0)+\mu(X\cap f_2^{-1}(0),0)-\tau(X\cap f_2^{-1}(0),0)$ .
                                             \end{center}
Hence,
                               \begin{eqnarray*}
                               % \nonumber % Remove numbering (before each equation)
                                Ch_{X\cap f_2^{-1}(0),0}\{df_1\}&=& ind_{X\cap f_2^{-1}(0),0}\{df_1\}-ind_{X\cap f_2^{-1}(0),0}\{L\}\\
                               %&=& \mu(X\cap g_2^{-1}(0),0)+\mu(X\cap g_2^{-1}(0)\cap g_1^{-1}(0),0)\\&-&(\mu(X\cap g_2^{-1}(0),0)+\mu(X\cap g_2^{-1}(0)\cap L^{-1}(0),0))\\
                              % &=& \mu(X\cap g_2^{-1}(0)\cap g_1^{-1}(0),0)-\mu(X\cap g_2^{-1}(0)\cap L^{-1}(0),0)\\
                              % &=&\mu_{BR}^{-}(g_1,X\cap g_2^{-1}(0))+\tau(X\cap g_2^{-1}(0),0)-\mu_{BR}^{-}(L,X\cap g_2^{-1}(0))-\tau(X\cap g_2^{-1}(0),0)\\
                               &=&\mu_{BR}^{-}(f_1,X\cap f_2^{-1}(0))-\mu_{BR}^{-}(L,X\cap f_2^{-1}(0)).\qed
                               \end{eqnarray*}

We can also use Proposition \ref{propositioncentral} and Theorem \ref{theorem central coleção de uma forma} to obtain the following formula.

\begin{corollary}\label{corolariocentral}
Let $(X,0)\subset(\mathbb{C}^n,0)$ be an ICIS and $f=(f_1,f_2):(X,0)\rightarrow(\mathbb{C}^2,0)$ a map such that $X\cap f_2^{-1}(0)$ and $X\cap f_2^{-1}(0)\cap f_1^{-1}(0)$ are also ICIS and $f_1:(X,0)\rightarrow(\mathbb{C},0)$ is tractable at the origin with respect to a good stratification $\mathcal{V}$ of $X\cap f_2^{-1}(0)$ relative to $f_2$. For the $1$-form $\{df_1\}$ and a generic linear form $L$,
  \begin{eqnarray*}
       \mu_{BR}^{-}(f_1,X\cap f_2^{-1}(0)) - \mu_{BR}^{-}(L,X\cap f_2^{-1}(0))=Eu_{X\cap f_2^{-1}(0)}(0)-1+Ch_{X\cap f_2^{-1}(\alpha),x_0}\{df_1\}.
    \end{eqnarray*}
\end{corollary}
\noindent\textbf{Proof.} By Proposition \ref{propositioncentral} and Theorem \ref{theorem central coleção de uma forma},
                               \begin{eqnarray*}
                               % \nonumber % Remove numbering (before each equation)
                                  \mu_{BR}^{-}(f_1,X\cap f_2^{-1}(0)) - \mu_{BR}^{-}(L,X\cap f_2^{-1}(0)) &=& Ch_{X\cap f_2^{-1}(0),0}\{df_1\} \\
                                  &=& Eu_{X\cap f_2^{-1}(0)}(0)-1+Ch_{X\cap f_2^{-1}(\alpha),x_0}\{df_1\},
                                                                 \end{eqnarray*}
\noindent where $x_0\in X\cap f_2^{-1}(\alpha). \qed$

\subsection{An example of complete intersections with non isolated singularity}

Let $(X,0)\subset(\mathbb{C}^n,0)$ be a reduced analytic variety determined by $\phi:(\C^{n},0)\to(\C^{k},0)$. Considering the suspension of $(X,0)$ in $\C^{n+t}$ we have $\tilde{X}=X\times\mathbb{C}^t.$ Let $(x_1,\ldots,x_n,y_1,\ldots,y_t)$ be the coordinate system in $\mathbb{C}^n\times\mathbb{C}^t.$ In \cite{lima2021relative}, the authors proved that %if $X$ is a hypersurface with isolated singularity at the origin, then
the logarithmic characteristic variety $LC(\tilde{X})$ of $\tilde{X}$ is related to the logarithmic characteristic variety $LC(X)$ of $X$ by the isomorphism $LC(\tilde{X})\simeq LC(X)\times\mathbb{C}^t$. Also, they proved that if $LC(X)$ is Cohen-Macaulay then $LC(\tilde{X})$ is Cohen-Macaulay. Following the ideas of \cite{lima2021relative}, we  know that %obtain that if $(X,0)\subset(\mathbb{C}^n,0)$ and $\tilde{X}=X\times\mathbb{C}^t$, then
$\Theta_{\tilde{X}}=\mathcal{O}_{n+t}\Theta_{X}+ \langle\frac{\partial}{\partial x_1},\dots,\frac{\partial}{\partial x_t}\rangle\subset\Theta_{n+t}.$ %where $I_{\tilde{X}}$ denotes the ideal associated to $\tilde{X}.$
 Hence the relative logarithmic characteristic variety $LC(\tilde{X})^{-}$ of $\tilde{X}$ is related to the relative logarithmic characteristic variety $LC(X)^{-}$ of $X$ by the isomorphism $LC(\tilde{X})^{-}\simeq LC(X)^{-}\times\mathbb{C}^t$. Therefore, if $LC(X)^{-}$ is Cohen-Macaulay, then $LC(\tilde{X})^{-}$ is Cohen-Macaulay.

Suppose additionally that $X$ is holonomic with logarithmic stratification $X_0=\mathbb{C}^n\setminus X, X_1, \ldots, X_{k}, X_{k+1}=\{0\}.$ Then the suspension $\tilde{X}=X\times\mathbb{C}^t$ of $X$ is holonomic and its logarithmic stratification is \begin{center}
  $\tilde{X}_0=X_0\times\mathbb{C}^t, \tilde{X}_1=X_1\times\mathbb{C}^t,\ldots, \tilde{X}_k=X_k\times\mathbb{C}^t, \tilde{X}_{k+1}=\{0\}\times\mathbb{C}^t.$
\end{center}
Considering $f\in\mathcal{O}_n, h\in\mathcal{O}_t$ and $F\in\mathcal{O}_{n+t}$, such that $F(x,y)=f(x)+h(y)$, $\mu_{BR}^{-}(f,X)<\infty$ and $\mu(h)<\infty.$ Hence, $\mu^{-}_{BR}(F,\tilde{X})<\infty$ and \begin{center}
                                         $\mu^{-}_{BR}(F,\tilde{X})=\mu(h)\mu_{BR}^{-}(f,X),$
                                       \end{center}
\noindent which is a consequence of the characterization of $\Theta_{\tilde{X}}$ and \cite[Proposition 5.1]{lima2021relative}. Let us denote by $n_{\alpha}$ the number of Morse critical points of a Morsification $f_t$ of $f$ in $X_{\alpha}$ and by $m_{\alpha}$ the geometrical multiplicity of its correspondent irreducible component $Y_{\alpha}$ in $LC(X).$

We also denote by $\tilde{n}_{\alpha}$ the number of Morse critical points of a Morsification $F_t$ of $F$ in $\tilde{X}_{\alpha}$ and by $\tilde{m}_{\alpha}$ the geometrical multiplicity of its correspondent irreducible component $\tilde{Y}_{\alpha}$ in $LC(\tilde{X}).$

\begin{proposition}
With the above conditions, $\tilde{n}_{\alpha}=\mu(h)n_{\alpha}$ and $\tilde{m}_{\alpha}=m_{\alpha}.$
\end{proposition}
\noindent\textbf{Proof.} Let $h_t$ be a Morsification of $h.$ Then $F_t=f_t+h_t$ is a Morsification of $F.$ By  \cite[Proposition 5.15]{bruce1988critical} and \cite[Theorem]{sebastiani1971resultat},
\begin{center}
 $\tilde{n}_0=\mu(F)=\mu(f)\mu(h)=\mu(h)n_0.$
  \end{center}

Let $\alpha=1,\ldots,k$ and $x\in X_{\alpha}.$ If $x$ is a critical point of $f_{t}|_{X_{\alpha}}$, then for all $y\in\Sigma h_t, (x,y)\in\tilde{X}_{\alpha}$ is a critical point of $F_{t}|_{X_{\alpha}}.$ Indeed, the matrix $\left[\begin{array}{rcr}
df_{t}\\d\phi
\end{array}\right]$ does not have maximum rank in $x$ and, for all $y\in\mathbb{C}^t,$ the matrix $\left[\begin{array}{rcr}
df_{t} & dh_t\\d\phi & 0
\end{array}\right]_{(k+1)\times n}$ has maximum rank if, and only if, $y\not\in\Sigma h_t.$ Hence, we obtain that $\tilde{n}_{\alpha}\geqslant\mu(h)n_{\alpha}.$ On the other hand, if $(x,y)\in\tilde{X}_{\alpha}$ is a critical point of $F_{t}|_{X_{\alpha}},$ the matrix $\left[\begin{array}{rcr}
df_{t} & dh_t\\d\phi & 0
\end{array}\right]_{(k+1)\times n}$ has no maximum rank in $(x,y),$ that is,  $\left[\begin{array}{rcr}
df_{t}\\d\phi
\end{array}\right]$ has no maximum rank in $x$ and $dh_t(y)=0.$ So, $x$ is a Morse critical point of $f_{t}|_{X_{\alpha}}$ and $y\in \Sigma h_t.$ Then we conclude that $\tilde{n}_{\alpha}=\mu(h)n_{\alpha}.$ Now consider the stratum $\tilde{X}_{k+1}=X_{k+1}\times\mathbb{C}^t=\{0\}\times\mathbb{C}^t.$ A point $(0,y)\in\tilde{X}_{k+1}$ is a critical point of $F_{t}|_{X_{k+1}}$ if, and only if, $y$ is a Morse critical point of $h_t.$ Hence, $\tilde{n}_{k+1}=\mu(h).$
To finish the proof, notice that $\tilde{m}_{\alpha}=m_{\alpha}=1,$ for all $\alpha=1,\ldots,k,$ since $X_{\alpha}$ and $\tilde{X}_{\alpha}$ are regular strata. Also, because $LC(X)$ is Cohen-Macaulay, $\mu_{BR}^{-}(f,X)=\sum_{\alpha=1}^{k+1}m_{\alpha}n_{\alpha}$ and \begin{eqnarray*}
                % \nonumber % Remove numbering (before each equation)
                  \mu_{BR}^{-}(F,\tilde{X}) &=& \mu(h)\mu_{BR}^{-}(f,X) \\
                   &=&\mu(h) \sum_{\alpha=1}^{k+1}m_{\alpha}n_{\alpha} \\
                   &=& \sum_{\alpha=1}^{k+1}m_{\alpha}\tilde{n}_{\alpha}.
                \end{eqnarray*}
We conclude that $\tilde{m}_{k+1}=m_{k+1}.$\qed

\begin{corollary}
Let $(X,0)\subset(\mathbb{C}^n,0)$ be an ICIS determined by the map $(\phi_1,\ldots,\phi_k):(\mathbb{C}^n,0)\rightarrow(\mathbb{C}^k,0)$. Let $f_2\in\mathcal{O}_n$ be a function-germ with isolated singularity at the origin in $X$ and $\tilde{f}_2\in\mathcal{O}_{n+t}$ be a function-germ such that $\tilde{f}_2(x_1,\ldots,x_n,y_1,\ldots,y_t)=f_2(x_1,\ldots,x_n).$ If $f_1\in\mathcal{O}_{n+t}$ such that $f_1|_{\tilde{X}\cap\{\tilde{f}_2=0\}\setminus\{0\}}$ has isolated singularity at the origin, then

\begin{center}
 $\mu^{-}_{BR}(f_1,\tilde{X}\cap\{\tilde{f}_2^{-1}(0)\})=Ch_{\tilde{X}\cap\{\tilde{f}_2^{-1}(0)\},0}\{df_{1}\}+m_{k+1}n_{k+1}.$
\end{center}

In addition, if there are holomorphic functions $f\in\mathcal{O}_{n}$ and $h\in\mathcal{O}_{t}$ such that $f_1=f(x)+h(y)$, the last equality can be written as the following:
\begin{center}
 $\mu^{-}_{BR}(f_1,\tilde{X}\cap\{\tilde{f}_2^{-1}(0)\})=Ch_{\tilde{X}\cap\{\tilde{f}_2^{-1}(0)\},0}\{df_{1}\}+m_{k+1}\mu(h).$
\end{center}
%where $\{\omega^{(i)}_{j}\}=\big\{\{df_1,d\tilde{f}_2\},\{l_1,\ldots,l_{n+t-k}\}\big\}$ and $l_j$ are generic linear forms.
\end{corollary}

We give next examples of holonomic varieties such that the relative logarithmic characteristic variety is Cohen-Macaulay and consequently our results can be applied.

\begin{example} {\bf Complete intersections with non isolated singularity.}

Let $(X,0)$ be the hypersurface with non-isolated singularity determined by $$\phi:(\C^{4},0)\to(\C,0)\textup{ given by }\phi(x,y,z,w)=x^{2}-z^{3}.$$ Then $(X,0)$ has non isolated singularity. Considering $f:(X,0)\to(\C^{2},0)$  given by
$$
f(x,y,z,w)=(f_{1},f_{2})=(x^{3}+x^{2}y^{2}+y^{7}+z^{3}+w^{2},y^{2}+w),
$$
then  $X\cap f_{2}^{-1}(0)$ defines a complete intersection with non-isolated singularity and the singular set of $X\cap f_{2}^{-1}(0)$, $\Sigma_{X\cap f_{2}^{-1}(0)}=\{(0,y,0,w); y^{2}+w=0\}$ is a curve.

 We use the software {\sc Singular} to calculate
$$
\Theta_{X\cap f_{2}^{-1}(0)}=\langle e_{2}-2ye_{4},\;\phi e_{4},\; f_{2}e_{4}, 3xe_{1}+2ze_{3},\;3z^{2}e_{1}+2xe_{3},\; f_{2}e_{1},\; f_{2}e_{3}\rangle.
$$
Then considering a system of coordinates in $(x,y,z,w,p_{1},p_{2},p_{3},p_{4})\in T^{*}\C^{4}$ we have that $LC(X\cap f_{2}^{-1}(0))^{-}=V(I^{-})$ where
$$
I^{-}=\langle p_{2}-2yp_{4},\;\phi p_{4},\; f_{2}p_{4}, 3xp_{1}+2zp_{3},\;3z^{2}p_{1}+2xp_{3},\; f_{2}p_{1},\; f_{2}p_{3}\rangle+\langle \phi, f_{2}\rangle.
$$
Since $\dim\mathcal{O}_{8}/I^{-}=depth(\mathcal{O}_{8}/I^{-})=4$, it follows that $LC(X\cap f_{2}^{-1}(0))^{-}$  is Cohen-Macaulay, and consequently we can apply our results.%the Theorem \ref{teoremacentral}. Therefore
%$$
%Ch_{X\cap g_{2}^{-1}(0)}\{\omega_{j}^{(i)}\}=\mu_{BR}^{-}(g_{1},X\cap g_{2}^{-1}(0))=15
%$$
%where $\{\omega_{j}^{(i)}\}=\{\{dg_{1},\;dg_{2}\},\{l_{1},\;l_{2}\}\}$, with $l_{1},l_{2}:(\C^{4},0)\to(\C,0)$  linear generic.
\end{example}

\begin{example}{\bf Isolated determinantal singularity (IDS)}

Let $(X,0)$ be the variety determined by $\phi:(\C^{4},0)\to(\C^{2},0)$ given by
$$
\phi(x,y,z,w)=(\phi_{1},\phi_{2})=( xz-y^2,\; xw-yz).
$$
Then $(X,0)$ is a complete intersection with non-isolated singularity.
Let $f:(X,0)\to(\C^{2},0)$ be given by
$$
f(x,y,z,w)=(f_{1},f_{2})=(x^{3}+yz+w^{4}, yw-z^{2}).
$$
Then $X\cap f_{2}^{-1}(0)$ is the IDS determined by
$$
I_{2}\begin{pmatrix}
x&y&z\\
y&z&w
\end{pmatrix}=\langle xz-y^2,\; xw-yz,\; yw-z^2 \rangle.
$$
The singular set of $(X\cap f_{2}^{-1}(0),0)$, $\Sigma _{X\cap f_{2}^{-1}(0)}=\{(0,0,0,0)\}$.

 We use the software {\sc Singular} to calculate $\Theta_{X\cap f_{2}^{-1}(0)}$ and it is generated by the columns of the matrix
$$
\begin{pmatrix}
0  &3x  &0   &3y & 0                 &-3xw\\
x  &y    &y   &2z & 3xw            &2z^{2}-3yw\\
2y&-z   &2z &w  & 2z^2+4yw &zw\\
3z&-3w&3w&0   &  9zw           &3w^{2}
\end{pmatrix}
$$
 and the vector fields
$$
 \langle \phi_{1},\; \phi_{2},\; f_{2} \rangle\Theta_{4}.
 $$
Then the ideal $I^{-}$ that defines the variety $LC(X\cap f_{2}^{-1}(0))^{-}$ is given by
\begin{align*}
I^{-}=\langle&xp_{2}+2yp_{3}+3zp_{4},\; 2xp_{1}+yp_{2}-wp_{4},\; yp_{2}+2zp_{3}+3wp_{4},\\
 &3yp_{1}+2zp_{2}+wp_{3},\;-3xwp_{1}+(2z^{2}-3yw)p_{2}+zwp_{3}+3w^{2}e_{4}\rangle+ I,
 \end{align*}
 where $I=\langle\phi_{1},\; \phi_{2},\; g_{2} \rangle.$

Since $\dim\mathcal{O}_{8}/I^{-}=depth(\mathcal{O}_{8}/I^{-})=4$ we have that $LC(X\cap f_{2}^{-1}(0))^{-}$ is Cohen-Macaulay and we can apply our results. %the Theorem \ref{teoremacentral}. Therefore
%$$
%Ch_{X\cap g_{2}^{-1}(0)}\{\omega_{j}^{(i)}\}=\mu_{BR}^{-}(g_{1},X\cap g_{2}^{-1}(0))=16,
%$$
%where $\{\omega_{j}^{(i)}\}=\{\{dg_{1},\;dg_{2}\},\{l_{1},\;l_{2}\}\}$, with $l_{1},l_{2}:(\C^{4},0)\to(\C,0)$  linear generic.
\end{example}

\begin{example}{\bf Determinantal variety with non isolated singularity}

Let $(X,0)$ be the variety determined by $\phi:(\C^{4},0)\to (\C^{2},0)$ given by
$$
\phi(x,y,z,w)=(\phi_{1},\phi_{2})=(xz,\;xw).
$$
We observe that $(X,0)$ is not a complete intersection.

Considering $f:(X,0)\to(\C^{2},0)$ given by
$$
f(x,y,z,w)=(f_{1},f_{2})=(x^{2}+y^{3}+zw+w^{2}, zy).
$$
Then $(X\cap f_{2}^{-1}(0),0)$ is the determinantal variety given by
 $$
 I_{2}\begin{pmatrix}
x&0&y\\
0&z&w
\end{pmatrix}=\langle xz,\; xw,\; zy \rangle.
$$
The singular set of $(X\cap f_{2}^{-1}(0),0)$, $\Sigma X=\{(0,y,0,0)\}\cup\{0,0,0,w\}$. \\
We use the software {\sc Singular} to calculate $\Theta_{X\cap f_{2}^{-1}(0)}$ and it is generated by the columns of the matrix
$$
\begin{pmatrix}
0 &x  &0  &0 &0 &0 &yz &0  \\
x &y  &y  &0 &0 &0 &0  &0  \\
0 &-z &0  &0 &z &0 &0  &xw  \\
0 &-w&0  &z &0 &w &0  &0
\end{pmatrix}.
$$
Then the ideal $I^{-}$ that defines the variety $LC(X)^{-}$ is given by
\begin{align*}
I^{-}=\langle& xp_{2},\; xp_{1}+yp_{2}-zp_{3}-wp_{4},\; yp_{2},\; zp_{4},\; zp_{3},\; wp_{4},\; yzp_{1},\; xwp_{3}\rangle+ I,
 \end{align*}
 where $I=\langle \phi_{1},\; \phi_{2},\; f_{2} \rangle$.

Since $\dim\mathcal{O}_{8}/I^{-}=depth(\mathcal{O}_{8}/I^{-})=4$ we have that $LC(X\cap f_{2}^{-1}(0))^{-}$ is Cohen-Macaulay. Therefore we can apply our results. %the Theorem \ref{teoremacentral} and
%$$
%Ch_{X\cap g_{2}^{-1}(0)}\{\omega_{j}^{(i)}\}=\mu_{BR}^{-}(g_{1},X\cap g_{2}^{-1}(0))=9,
%$$
%where $\{\omega_{j}^{(i)}\}=\{\{dg_{1},\;dg_{2}\},\{l_{1},l_{2},l_{3}\}\}$, with $l_{1},l_{2},l_{3}:(\C^{4},0)\to(\C,0)$  linear generic.
\end{example}

In the previous examples the relative logarithmic characteristic is Cohen Macaulay. Since these varieties are holonomic we have that our results are true for these varieties.

\section{Invariants of map-germs into the plane}

It was mentioned in the last section that Ebeling and Gusein-Zade give in \cite{chernobstructions} an algebraic formula for the Chern number of a collection of 1-forms $\{\omega_{j}^{(i)}\}$. When $(X,0)\subset(\C^{n},0)$ is a smooth variety, this characterization reduces to \begin{equation}\label{chern suave}
Ch_{X,0}\{\omega_{j}^{(i)}\}=ind_{X,0}\{\omega_{j}^{i}\}=\dim_{\C}\frac{\mathcal{O}_{n}}{I_{\{\omega_{j}^{(i)}\}}},
\end{equation} where $I_{\{\omega_{j}^{i}\}}=J(\omega ^{(1)})+...+J(\omega^{(s)}),$ with $J(\omega^{i})$ the ideal generated by the minors of maximum order of the jacobian matrix of $\omega^{i}:(\C^{n},0)\to(\C^{n-k_{i}+1})$.

%Motivados pela referência , obtemos a seguir uma fórmula que relaciona o número de Chern com a soma do número de cúspide de uma perturbação de f  e o número de Milnor de f_1. O resultado corrige a fórmula ??? da Proposição ??? do artigo ???

We were motivated by \cite{brasselet2010euler} to consider a finitely determined map germ $f=(f_{1},f_{2}):(\C^{2},0)\to(\C^{2},0)$ and we obtain a relation between the number of cusps of a perturbation of $f$ and $Ch_{\C^{2},0}\{\{df_1,df_{2}\},\{df_{1},d\Delta\}\}$. This result corrects Proposition 4.1 in \cite{brasselet2010euler}.

Consider the following collections of $1$-forms: $\big\{\eta_1\big\}=\big\{\{df_{1},df_{2}\},\{df_{1},d\Delta\}\big\}$ and $\big\{\eta_2\big\}=\big\{\{df_{1},df_{2}\},\{df_{2},d\Delta\}\big\},$ where $\Delta$ is the determinant of the jacobian matrix of $f$.

\begin{proposition}\label{cuspides variedade suave}
Let $f=(f_{1},f_{2}):(\C^{2},0)\to(\C^{2},0)$ be a finitely determined map germ. Then $$Ch_{\C^{2},0}\big\{\eta_1\big\}=c(f)+\mu(f_{1}),$$ where $c(f)$ is the number of cusps of a generic perturbation of $f$, and $\mu(f_{1})$ is the Milnor number of function germ $f_{1}$.
\end{proposition}

\begin{proof}
%We know that \begin{equation} \label{chern suave} Ch_{\C^{2},0}{\omega_{j}^{(i)}}=ind_{\C^{2},0}{\omega_{j}^{(i)}}=\dim_{\C}\frac{\mathcal{O}^{2}_{2}}{I_{{\omega_{j}^{(i)}}}},\end{equation} where $I_{\omega_{j}^{(i)}}$ is the ideal generated by determinants $$\Delta=\left| \begin{array}{rcr}
%df_{1}\\df_{2}
%\end{array} \right|=\left| \begin{array}{rcr}
%\frac{\partial f_{1}}{\partial x}&\frac{\partial f_{1}}{\partial y}\vspace{0.1cm}\\
%\frac{\partial f_{2}}{\partial x}&\frac{\partial f_{2}}{\partial y}
%\end{array} \right| \textup{ and } \left| \begin{array}{rcr}
%df_{1}\\d\Delta
%\end{array} \right| = \left|\begin{array}{rcr}
% \frac{\partial f_{1}}{\partial x}&\frac{\partial f_{1}}{\partial y}\vspace{0.1cm}\\
%\frac{\partial \Delta}{\partial x}&\frac{\partial \Delta}{\partial y}
%\end{array} \right|,$$
%
%see \cite[p.5 Remark 1]{chern obstructions}. Then
We consider the following sequence

\begin{equation} \label{sequencia exata 1}0\longrightarrow\frac{\mathcal{O}_{2}}{Jf_{1}}\stackrel{\psi}\longrightarrow\frac{\mathcal{O}_{2}}{\left\langle\Delta,\left| \begin{array}{rcr}
df_{1}\\d\Delta
\end{array} \right| \right\rangle}\stackrel{\pi}\longrightarrow\frac{\mathcal{O}_{2}}{J(f,\Delta)}\longrightarrow 0,
\end{equation}
where $J(f,\Delta)$ is the ideal generated by minors of maximum order of the jacobian matrix of $(f,\Delta):(\C^{2},0)\to(\C^{3},0)$ and $\psi$ is given by the product by the determinant $$\left| \begin{array}{rcr}
df_{2}\\d\Delta
\end{array} \right|=\left|\begin{array}{rcr}
 \frac{\partial f_{2}}{\partial x}&\frac{\partial f_{2}}{\partial y}\vspace{0.1cm}\\
\frac{\partial \Delta}{\partial x}&\frac{\partial \Delta}{\partial y}
\end{array} \right|,$$ and $\pi$ is the projection.

In \cite[Lemma 1.7]{gaffneymond} Gaffney and Mond prove that the sequence (\ref{sequencia exata 1}) is exact. Then with equality (\ref{chern suave}) $$Ch_{\C^{2},0}\big\{\eta_1\big\}=\dim_{\C}\frac{\mathcal{O}_{2}}{\left\langle\Delta,\left| \begin{array}{rcr}
df_{1}\\d\Delta
\end{array} \right|\right\rangle}=c(f)+\mu(f_{1}).$$

\end{proof}
In the next result we extend Proposition \ref{chern suave} to map germs defined on ICIS. %but first we need to introduce some notation.

Let $f=(f_{1},f_{2}):(X,0)\to(\C^{2},0)$ be $\mathcal{A}$-finite,  see \cite[Definição 2.1]{juanjobrunatomazella3}. The number of cusps of a stabilization of $f$  is given by $$c(f\big|_{X})=\dim_{\C}\frac{\mathcal{O}_{n}}{I_{X}+J(\phi,f,\Delta)},$$
where $\Delta$ is the determinant of the jacobian matrix of $(\phi,f):(\C^{n},0)\to(\C^{n},0)$, and $I_{X}$ is generated by the coordinates of $\phi$, \cite[Proposition 3.3]{juanjobrunatomazella3}.

 The ideal $J(\phi,f,\Delta)$ is generated by $n$ elements $$\beta_{1},\;...,\;\beta_{n},\;\beta_{n+1},$$
where $\beta_{i}$ is the determinant of the matrix obtained by the jacobian matrix of $(\phi,f,\Delta)$ deleting the i-th row. Note that $\beta_{n+1}=\Delta.$

\begin{proposition}\label{cuspides icis}
Let $(X,0)$ be an ICIS determined by map germ $\phi=(\phi_{1},...,\phi_{n-2}):(\C^{n},0)\to(\C^{n-2},0)$, $f=(f_{1},f_{2}):(X,0)\to(\C^{2},0)$ an $\mathcal{A}$-finite map germ  and $\big\{\eta_1\big\}=\big\{\{df_{1},df_{2}\},\{df_{1},d\Delta\}\big\},$, where $\Delta$ is the determinant of the jacobian matrix of the map germ $(\phi,f)$. In this case
\begin{equation}\label{chern de uma ICIS}
Ch_{X,0}\big\{\eta_1\big\}=c(f\big|_{X})+\mu(f_{1}\big |_{X})-ind_{X,0}\{l_{j}^{(i)}\},
\end{equation}
where $\{l_{j}^{(i)}\}$ is a collection of generic linear functions on $\C^{n}$, $c(f\big|_{X})$ is the number of cusps of a generic perturbation of $f:(X,0)\to(\C^{2},0)$, and $\mu(f_{1}\big|_{X})$ is the Milnor number of the function germ $f_{1}$ restricted to $(X,0)$, see \cite[Remark 3.4]{juanjobrunatomazella2}.
\end{proposition}
\begin{proof}
Inspired by the smooth case we consider the following sequence
\begin{equation}\label{sequencia exata 2}0\longrightarrow\frac{\mathcal{O}_{n}^{n-1}}{I_{X}\mathcal{O}_{n}^{n-1}+Im(d(\phi,f))}\stackrel{\Psi}\longrightarrow\frac{\mathcal{O}_{n}}{I_{X}+\left\langle\Delta,\beta_{n} \right\rangle}\stackrel{\pi}\longrightarrow\frac{\mathcal{O}_{n}}{I_{X}+J(\phi,f,\Delta)}\longrightarrow 0,
\end{equation}

where $\Psi(a_{1},...,a_{n-1})=\sum_{i=1}^{n-1}a_{i}\beta_{i}$ and $\pi$ is the projection.
To conclude that (\ref{sequencia exata 2}) is exact we need to prove that $\Psi$ is a monomorphism and we use the Hilbert-Burch Theorem \cite{hilbertburch}.

Let $R=\mathcal{O}_{n}/I_{X}$ then the codimension of the ideal $J(\phi,f,\Delta)$ in $R$ is equal to 2 and the complex

\begin{equation}\label{Hilbert Burch no nosso caso}
0\longrightarrow R^{n}\stackrel{d(\phi,f,\Delta)}\longrightarrow R^{n+1}\stackrel{\lambda}\longrightarrow R \longrightarrow\frac{R}{J(\phi,f,\Delta)}\longrightarrow 0
\end{equation}

where $\lambda$ is the map whose components are the generators of $J(\phi,f,\Delta)$ with appropriate sign, is an exact sequence and $ker(\lambda)=Im(d(\phi,f,\Delta))$.

Returning to $\Psi$, if $(a_{1},...,a_{n-1})\in \ker(\Psi)$ we have $$\sum_{i=1}^{n-1}a_{i}\beta_{i}=0\textup{ in } \mathcal{O}_{n}/ (I_{X}+\langle\beta_{n},\beta_{n+1}\rangle).$$
Therefore there exist $a_{n}, a_{n+1} \in \mathcal{O}_{n}$ such that $\sum_{i=1}^{n+1}a_{i}\beta_{i}\in I_{X}$, and $$(a_{1},...a_{n-1},a_{n},a_{n+1})\in \ker(\lambda)=Im(d(\phi,f,\Delta)).$$

It proves that sequence (\ref{sequencia exata 2}) is exact and $$ind_{X}\big\{\eta_1\big\}=c(f\big|_{X})+\dim_{\C}\frac{\mathcal{O}_{n}^{n-1}}{I_{X}\mathcal{O}_{n}^{n-1}+Im(d(\phi,f_{1}))}.$$

Since $c(f)$ is finite and we can assume that $f_{1}\big|_{X}$ has isolated singularity follows from of the \cite[Theorem C.8]{mond2020singularities} that $\dim_{\C}\mathcal{O}_{n}^{n-1}/(I_{X}\mathcal{O}_{n}^{n-1}+Im(d(\phi,f_{1})))=\dim_{\C}\mathcal{O}_{n}/(I_{X}+J(\phi,f_{1}))=\mu(f_{1}\big|_{X}).$ Therefore the equality (\ref{chern de uma ICIS}) follows from the equality (\ref{relação para o chern de uma ICIS}).

\end{proof}

Notice that Proposition \ref{cuspides icis} can be proved using the colletion of $1$-forms $\big\{\eta_2\big\}=\big\{\{df_{1},df_{2}\},\{df_{2},d\Delta\}\big\}.$  In this case, the equality obtained is \begin{equation}\label{chern de uma ICIS}
Ch_{X,0}\big\{\eta_2\big\}=c(f\big|_{X})+\mu(f_{2}\big |_{X})-ind_{X,0}\{l_{j}^{(i)}\}.
\end{equation}

The next result gives a formula to compute the number of cusps using the relative Bruce-Roberts number.
%In \cite[Proposition 4]{chern obstructions} Ebeling and Gusein-Zade prove that the difference $$ind_{X,0}\{\omega_{j}^{(i)}\}-Ch_{X,0}\{\omega_{j}^{(i)}\}$$ does not depend on the collection
\begin{corollary}\label{diferença de cherns e diferença de milnors}
In the same conditions of Proposition \ref{cuspides icis},
$$
Ch_{X,0}\big\{\eta_1\big\}-Ch_{X,0}\big\{\eta_2\big\}=\mu(f_{1}\big|_{X})-\mu(f_{2}\big|_{X}).
$$
\end{corollary}

\begin{example} In this example we use the equality of the Proposition \ref{cuspides icis} to calculate the Chern number. Let $(X,0)$ be the isolated hypersurface singularity determined by $\phi(x,y,z)=x^3+x^2y^2+y^7+z^2$, and $f(x,y,z)=(f_{1}(x,y,z),f_{2}(x,y,z))=(y+z^2,x^2+xy+y^2)$.

Let $\big\{\eta_1\big\}=\{\{df_{1},df_{2}\},\{df_{1},d\Delta\}\}$ and
$\big\{\eta_2\big\}=\{\{df_{1},df_{2}\},\{df_2,d\Delta\}\}$ be the collections of 1-forms. Then
$$c(f\big|_{X})=9,\;\mu(f_{1}\big|_{X})=13, \mu(f_{2}\big|_{X})=18  \textup{ and } ind_{X,0}\{l_{j}^{(i)}\}=13,$$
where $\{l_{j}^{(i)}\}=\{\{x+y,x-y+3z\},\{x+y-z,x-y+5z\}\}$. Therefore
$$Ch_{(X,0)}\big\{\eta_1\big\}=9\textup{ and }Ch_{(X,0)}\big\{\eta_2\big\}=14.$$

\end{example}

We now present a relation between the difference of the Chern number of $X$ with a specific collection of 1-forms and the Chern numbers of $X\cap f_{1}^{-1}(0)$ and $X\cap f_{2}^{-1}(0)$.

\begin{proposition}\label{diferença dos cherns, dos milnors relativos e dos tjurinas}
Let $(X,0)$ be an ICIS determined by the map germ $\phi=(\phi_{1},...,\phi_{n-2}):(\C^{n},0)\to(\C^{n-2},0)$ and $f=(f_{1},f_{2}):(X,0)\to(\C^{2},0)$ an $\mathcal{A}$-finite map germ, such that $X\cap f_2^{-1}(0)$ and  $X\cap f_2^{-1}(0)\cap  f_1^{-1}(0)$ are ICIS. For the collections of $1$-forms $\big\{\eta_1\big\}$ and $\big\{\eta_2\big\},$
\begin{align*}
Ch_{X,0}\big\{\eta_1\big\}-Ch_{X,0}\big\{\eta_2\big\}=Ch_{X\cap f_{2}^{-1}(0)}\{df_{1}\}-Ch_{X\cap f_{1}^{-1}(0)}\{df_{2}\}+\mu_{BR}^{-}(l,X\cap f_{2}^{-1}(0))\\-\mu_{BR}^{-}(l,X\cap f_{1}^{-1}(0))+\tau(X\cap f_{2}^{-1}(0))-\tau(X\cap f_{1}^{-1}(0))
\end{align*}
\end{proposition}
\begin{proof}
By Corollary  \ref{diferença de cherns e diferença de milnors} it follows

\begin{equation}\label{diferença dos cherns em X}
Ch_{X}\big\{\eta_1\big\}-Ch_{X}\big\{\eta_2\big\}=\mu(f_{1}\big|_{X})-\mu(f_{2}\big|_{X}).
\end{equation}
By \cite[Theorem 2.2]{lima2021relative},

\begin{eqnarray}\label{diferença dos relativos}
\mu_{BR}^{-}(f_{1},X\cap f_{2}^{-1}(0))-\mu_{BR}^{-}(f_{2},X\cap f_{1}^{-1}(0))&=&\mu(X\cap f_{2}^{-1}(0))-\mu(X\cap f_{1}^{-1}(0)) \ \ \\ &+& \nonumber\tau(X\cap f_{1}^{-1}(0))-\tau(X\cap f_{2}^{-1}(0)).
\end{eqnarray}

With the previous equalities,

\begin{align*}
Ch_{X}\big\{\eta_1\big\}-Ch_{X}\big\{\eta_2\big\}&=\mu(f_{1}\big|_{X})-\mu(f_{2}\big|_{X})\\
&=\mu_{BR}^{-}(f_{1},X\cap f_{2}^{-1}(0))-\mu_{BR}^{-}(f_{2},X\cap f_{1}^{-1}(0))+\tau(X\cap f_{2}^{-1}(0))\\
&-\tau(X\cap f_{1}^{-1}(0))\\
\end{align*}

With the equality of Proposition \ref{theorem central coleção de uma forma} and the previous equality we obtain

\begin{align*}
Ch_{X,0}\big\{\eta_1\big\}-Ch_{X,0}\big\{\eta_2\big\}=Ch_{X\cap f_{2}^{-1}(0)}\{df_{1}\}-Ch_{X\cap f_{1}^{-1}(0)}\{df_{2}\}+\mu_{BR}^{-}(l,X\cap f_{2}^{-1}(0))\\-\mu_{BR}^{-}(l,X\cap f_{1}^{-1}(0))+\tau(X\cap f_{2}^{-1}(0))-\tau(X\cap f_{1}^{-1}(0))
\end{align*}
\end{proof}

\textbf{Remark.} Given an $A$-finite map-germ $f=(f_1,f_2):(X,0)\subset(\C^{n},0)\to (\C^{2},0)$, up to change of coordinates in the target, we can assume that each coordinate of $f$ is a generic element in the set  $\{f_{\alpha}|_X=(\alpha_1f_1+\alpha_2f_2)|_X, \alpha_i\in\C\}$. Hence, equalities $\mu_{BR}^{-}(l,X\cap f_{1}^{-1}(0))=\mu_{BR}^{-}(l,X\cap f_{2}^{-1}(0))$ and $\tau(X\cap f_{1}^{-1}(0))=\tau(X\cap f_{2}^{-1}(0))$ hold. Notice that a generic element in this set is a function $f_{\alpha}|_X,$ for which the numbers $\mu_{BR}^{-}(l,X\cap f_{i}^{-1}(0))$ and $\tau(X\cap f_{i}^{-1}(0))$ are minimal. In this case, the formula in Proposition \ref{diferença dos cherns, dos milnors relativos e dos tjurinas} reduces to \begin{align*}
Ch_{X,0}\big\{\eta_1\big\}-Ch_{X,0}\big\{\eta_2\big\}=Ch_{X\cap f_{2}^{-1}(0)}\{df_{1}\}-Ch_{X\cap f_{1}^{-1}(0)}\{df_{2}\}.
\end{align*}

\bibliography{GRS}
\bibliographystyle{abbrv}

\begin{comment}

%\bibliography{GRS}
%\bibliographystyle{abbrv}

\end{comment}
\end{document}